\documentclass{imanum}
\usepackage{graphicx}
\usepackage{color,xcolor}
\usepackage{amssymb}
\usepackage{amsmath}

\jno{drnxxx}
\def\R{\mathbb{R}}
\def\Z{\mathbb{Z}}

\def\N{\mathcal{N}}
\def\K{\mathcal{K}}

\def\L{\mathcal{L}}
\def\D{\mathcal{D}}
\def\k{\mathbb{K}}

\def\F{\mathcal{F}}
\def\O{\mathcal{O}}
\def\V{\mathcal{V}}

\def\F{\mathcal{F}}
\newcommand{\vx}{\bm{x}}
\newcommand{\vy}{\bm{y}}
\newcommand{\vk}{\bm{k}}
\newcommand{\vm}{\bm{m}}
\newcommand{\vn}{\bm{n}}
\newcommand{\vz}{\bm{z}}

\newcommand{\vxi}{\bm{\xi}}
\newcommand{\va}{\bm{a}}

\newcommand{\vg}{\bm{g}}
\newcommand{\vl}{\bm{l}}
\newcommand{\vu}{\bm{u}}
\newcommand{\vv}{\bm{v}}
\newcommand{\vs}{\bm{s}}
\newcommand{\vU}{\bm{U}}

\newcommand{\vi}{\bm{i}}

\newcommand{\beq}{\begin{eqnarray}}
\newcommand{\eeq}{\end{eqnarray}}
\newcommand{\beqs}{\begin{eqnarray*}}
\newcommand{\eeqs}{\end{eqnarray*}}

\begin{document}

\title{Stability and convergence analysis of high-order numerical schemes with DtN-type absorbing boundary conditions for nonlocal wave equations}
% Short title for running heads:
\shorttitle{Stability and convergence analysis for nonlocal wave equations}

\author{
\sc
%Jihong Wang\thanks{Email: jhwang@csrc.ac.cn}\\[2pt]
%Beijing Computational Science Research Center, Beijing 100193, China\\[6pt]
Jihong Wang\thanks{Email: jhwang@zhejianglab.com}\\[2pt]
Research Center for Applied Mathematics and Machine Intelligence, Zhejiang Lab, Hangzhou 311121, China\\[6pt]
{\sc and}\\[6pt]
Jerry Zhijian Yang\thanks{Email: zjyang.math@whu.edu.cn}
and Jiwei Zhang\thanks{Corresponding author. Email: jiweizhang@whu.edu.cn} \\[2pt]
School of Mathematics and Statistics, and Hubei Key Laboratory of Computational Science,
Wuhan University, Wuhan 430072, China
}
% Short list of authors for running heads:
\shortauthorlist{Wang, Yang and Zhang}

\maketitle

\begin{abstract}
% Body of abstract:
{The stability and convergence analysis of high-order numerical approximations for the one- and two-dimensional nonlocal wave equations on unbounded spatial domains are considered. We first use the quadrature-based finite difference schemes to discretize the spatially nonlocal operator, and apply the explicit difference scheme to approximate the temporal derivative to achieve a fully discrete infinity system. After that, we construct the Dirichlet-to-Neumann (DtN)-type absorbing boundary conditions (ABCs) to reduce the infinite discrete system into a finite discrete system. To do so, we first adopt the idea in [Du, Zhang and Zheng, \emph{Commun. Comput. Phys.}, 24(4):1049--1072, 2018 and Du, Han, Zhang and Zheng, \emph{SIAM J. Sci. Comp.}, 40(3):A1430--A1445, 2018] to derive the Dirichlet-to-Dirichlet (DtD)-type mappings for one- and two-dimensional cases, respectively. We then use the discrete nonlocal Green's first identity to achieve the discrete DtN-type mappings from the DtD-type mappings. The resulting DtN-type mappings make it possible to perform the stability and convergence analysis of the reduced problem. Numerical experiments are provided to demonstrate the accuracy and effectiveness of the proposed approach. }
% Keywords:
{nonlocal wave equation, artificial boundary method, absorbing boundary conditions, stability and convergence analysis, DtN-type map}
\end{abstract}

\section{Introduction}
\label{sec;introduction}

Recently, nonlocal models have received much attention in various research areas, such as the peridynamic theory of continuum mechanics \citep[see][]{silling2000reformulation}, image processing \citep[see, e.g.,][]{buades2005non,gilboa2009nonlocal, lou2010image}, biology \citep[see, e.g.,][]{painter2015nonlocal} and diffusion processes \citep[see, e.g.,][]{d2017nonlocal, ignat2007nonlocal}.
%Much development about nonlocal models has been summarized in \cite{du2019nonlocal}.
While most existing nonlocal models are formulated on bounded domains with volume constraints  \citep[see][]{emmrich2007analysis,du2013nonlocal,tian2013analysis,tian2014asymptotically,zhou2010mathematical}, the models on infinite domains are more reasonable when describing wave propagation in an exceedingly large sample  \citep[see, e.g.,][]{weckner2005the,weckner2005numerical}.
In this work, we consider the computation of the $d$-dimensional nonlocal wave equation with $d=1,2$, given as
\beq \label{NonModel}
\begin{aligned}
& \partial_t^2 u(\vx,t)+\L_{\delta} u(\vx,t) = f(\vx,t), && \vx\in \R^d,\\
& u(\vx,0) = \varphi(\vx),&& \vx\in \R^d,\\
& \partial_tu(\vx,0) = \psi(\vx),&& \vx\in \R^d,
\end{aligned}
\eeq
where the body force $f(\vx, t)$, the initial values $\varphi(\vx)$ and $\psi(\vx)$ are the given compactly supported functions, and the nonlocal operator $\L_{\delta}$ is defined as
\beq \label{L}
\L_{\delta} u(\vx) = \int_{B_{\delta}(\vx)}(u(\vx)-u(\vy))\gamma(\vx-\vy)d\vy.
\eeq
In the definition above,  {$B_{\delta}(\vx)$ is an interval ($d=1$) or a square ($d=2$) centered at $\vx$ with side length $2\delta$}, and the radial kernel function $\gamma(\bm{\alpha})$ satisfies
\beq
&&-\; \text{nonnegativity:}\; \gamma(\bm{\alpha})\ge 0,\; \bm{\alpha}\in\R^d,\\
%&&\text{symmetry:}\; \gamma(\bm{\alpha})=\gamma(-\bm{\alpha}),~ \bm{\alpha}\in\R^d,\label{ass1}\\
&&-\;\text{finite horizon:} \; \gamma(\bm{\alpha})=0, ~ \bm{\alpha}\in\R^d\backslash B_{\delta}(\bm{0}).\label{ass2}
\eeq
Here horizon parameter $\delta$ is used  to measure the range of nonlocal interaction.
%, and the symbol $|\cdot|$ denotes the maximum norm of a scalar or vectorial object here and hereafter.
%And let $\Vert\cdot\Vert$ denotes the standard Euclidean norm.
%Moreover, if we assume the following two conditions:
%\beq
%&& \text{homogeneous:} \; \exists~ \text{a function} \; \gamma_0, \; \text{such that} \; \gamma(\bm{\alpha},\bm{\beta})=\gamma_0(\bm{\alpha});\\ \label{assump1}
%&&\text{second moment condition:} \; 0<d=\frac12\int_{|\bm{\alpha}|\le\delta}\Vert\bm{\alpha}\Vert^2\gamma_0(\bm{\alpha})d\vs<\infty. \label{assump2}
%\eeq
Moreover, if the second-order moment of the kernel satisfies
\beq
\frac{1}{2}\int_{\bm{\alpha}\in B_{\delta}(\bm{0})} \|\bm{\alpha}\|^2 \gamma(\bm{\alpha})d\bm{\alpha}=d,  \label{assump2}
\eeq
%the problem \eqref{NonModel} has the unique solution  \citep[see, e.g.,][]{du2011mathematical}.
%\beq
%\frac{1}{2}\int_{|\bm{\alpha}|\le\delta}\bm{\alpha}_{i}^2\gamma(\bm{\alpha})d\bm{\alpha}=1, \quad i=1,\cdots,d, \label{assump2}
%\eeq
%\beq
%\frac{1}{2}\int_{|\bm{\alpha}|\le\delta}\bm{\alpha} \otimes \bm{\alpha} \gamma(|\bm{\alpha}|)d\bm{\alpha}=I_d,  \label{assump2}
%\eeq
%where $I_d$ is the $d\times d$ identity matrix,
then the nonlocal operator $\L_{\delta}$ converges to the classical Laplace operator
$-\Delta$ when $\delta\rightarrow 0$ \citep[see][]{du2018nonlocal,du2019asymptotically,du2011mathematical,du2019nonlocal}.
Consequently, as $\delta\rightarrow 0$, the solution of nonlocal model \eqref{NonModel} converges to that of the following local model
\begin{eqnarray} \label{LocModel}
\begin{aligned}
& \partial_t^2 u(\vx,t) -\Delta u(\vx,t) = f(\vx,t), && \vx\in \R^d,t >0,\\
& u(\vx,0) = \varphi(\vx),&& \vx\in \R^d,\\
& \partial_tu(\vx,0) = \psi(\vx), && \vx\in \R^d.
\end{aligned}
\end{eqnarray}

Several tools have been developed to solve problems defined on the unbounded domains, such as the artificial boundary method (ABM) \citep[see][]{han2013artificial}, perfectly matched layer method \citep[see][]{berenger1994perfectly}, infinite element or boundary element method \citep[see][]{ying1980infinite,yu1993mathematical} and so on.  Among the above successful approaches, we here use the ABM to deal with the problem \eqref{NonModel}. The key ingredient of ABM is to design appropriate absorbing/artificial boundary conditions (ABCs), also called transparent or nonreflecting boundary conditions in literatures, on the artificial boundaries satisfied by the solution of the original problem, which reduce the original unbounded problem to an initial-boundary-value problem on bounded computational domains of interest. The ideal ABCs can efficiently absorb/annihilate waves on artificial boundaries, and do not produce the reflected or nonphysical waves to disrupt the waves in the computational domain.

The ABM has  {been} well studied to solve local problems on unbounded spatial domains \citep[see][]{grote1995exact, hagstrom1999radiation, lubich2002fast, teng2003exact, givoli2004high,givoli1991non}.
%A typical approach of obtaining the ABCs is to solve an exterior problem combining some techniques, such as Laplace transform, Pad$\acute{\text{e}}$ expansion and so on. To get a numerical solution to the problem, one needs to discretize the differential ABC, which usually involves convolution operations. Consequently, discretized ABC has complicated form and is hard to guarantee the stability and optimal convergence order of the numerical scheme \cite{}.
%ABM for nonlocal problems is still in its infancy.
%there are few works to solve the nonlocal problems on unbounded domains due to the nonlocality of the model itself and the complexity of nonlocal boundaries.
For local problems, the well-posedness requires values of the solution along only the boundary of considered domain $\Omega$. For nonlocal problems,  it requires values of the solution over a layer with a thickness of $\delta$ outside of $\Omega$ due to the nonlocal interactions. This brings essential difficulties to the design of ABCs for nonlocal problems, compared with local problems \citep[see][]{zhang2021}.  Recently, much effort and great progress have  been made for nonlocal problems \citep[see][]{zheng2017numerical, zhang2017artificial, du2018nonlocal, du2018numerical, yan2020numerical, zheng2020stability, shojaei2020dirichlet, ji2021artificial, ji2021accurate, wang2021stability}.
 For 1D nonlocal diffusion equations, \citet{zhang2017artificial} derive the continuous Dirichlet-to-Neumann (DtN)-type ABCs (global in time) and high-order Pad\'e approximate ABCs (local in time).
 \citet{zheng2017numerical} construct the Dirichlet-to-Dirichlet (DtD)-type ABCs using the Laplace transform in the spatial direction; Furthermore, \citet{zheng2020stability} propose the discrete DtN-type ABCs for the stability and convergence analysis and develop a fast convolution algorithm to efficiently implement the ABCs. In addition, \citet{shojaei2020dirichlet} construct the approximated Dirichlet-type ABCs derived from exponential basis functions for both 1D, 2D and 3D cases.
 For nonlocal Schr\"odinger equations, {\citet{yan2020numerical}} construct the exact ABCs using the $z$-transform for the spatially discretized 1D system. \citet{ji2021artificial} develop an
 exact boundary conditions by accurately computing the Green's functions of the semi-discrete nonlocal Schr\"odinger equations.
As for nonlocal wave equations given by \eqref{NonModel}, \citet{du2018nonlocal, du2018numerical} construct the DtD-type ABCs using the spatial Laplace transform for 1D case and the idea of integral equation method for 2D case, respectively, but the  stability and convergence analysis of the proposed schemes remains open. %The main reason is that, unlike the stability analysis for nonlocal diffusion equations \cite{zheng2020stability} or nonlocal Schr\"odinger equations \cite{wang2021stability},
%in the analysis of the stability of wave equations with time second derivative term by energy method, the inner product of DtN operator and the first time derivative of $u$ is required to be non-negative instead of the non-negative definite property of boundary DtN operator.
%when  the energy method is used to analyze the stability of the wave equations which includes the second-order derivative term of time, what is needed is no longer the nonnegative definite property of the boundary DtN operator, but the nonnegativity of inner product between the DtN operator and the first-order time derivative, which is usually not true.

The aim of this work is to construct numerical schemes with the rigorous stability and convergence analysis for nonlocal wave equation \eqref{NonModel} on unbounded domains.
A typical procedure of solving the problem on an unbounded domain by the ABM is first to derive suitable ABCs to restrict problem on a bounded domain, then approximate the reduced initial-boundary-value problem. The resulting ABCs for the continuum model usually involve convolution operations and other complicated forms, therefore, are hard to be approximated without loss of accuracy of the whole  numerical scheme, not to mention their stability and convergence analysis. An alternative procedure is first of all to fully discretize the original problem on the unbounded domain, and then directly construct the exact discrete ABCs for the fully discrete infinite system. %\rd{This strategy can not only completely eliminate any numerical reflections on boundaries of the infinite discrete system, but numerical stability is often also better-behaved \cite{arnold1998numerically, arnold2003discrete, li2018efficient}. }

In this work, we adopt the second strategy to discretize the continuum models into infinite discrete systems over the whole space. The explicit finite difference (FD) scheme is used to approximate the  temporal derivative. And the spatially discrete schemes here we used are the quadrature-based difference schemes, which can be arbitrarily high-order. After that, we apply the DtD-type ABCs developed in \citet{du2018nonlocal, du2018numerical} to reduce the infinite system to a finite system on the bounded domain. The resulting discrete DtD-type ABCs are exact and are tractable for practical implementations, but it is hard to obtain their stability and convergence analysis. To this end, we further construct the discrete DtN-type ABCs based on the discrete nonlocal Green's first identity. The DtN-type ABCs is useful to establish the stability and convergence of the reduced finite system.
%When using the energy method to analyze the stability of the numerical scheme, we combining the
Using the energy method, we prove that under the CFL condition of the nonlocal case, the proposed numerical scheme has an optimal convergence order of $\O(\tau^2+h^q)$, where $\tau$ and $h$ denote the time step size and spatial mesh size.  %To estimate the boundary DtN term, we consider the exterior problem and utilize the information of the whole equation.

 The paper is organized as follows. In section 2, a fully discrete scheme is presented to approximate the nonlocal wave equation \eqref{NonModel} to obtain the infinite discrete system. In section 3, the DtN-type ABCs are constructed based on the DtD-type ABCs, which reduce the infinite discrete system to a finite discrete systems. In section 4,  the stability and convergence of the proposed numerical schemes are analyzed, and numerical experiments are provided to demonstrate our theoretical analysis in section 5. The conclusion is drawn in section 6.

\section{Fully discrete wave system}
 %In this section, we will discretize the nonlocal operator \eqref{NonOp} and Laplace operator in \eqref{LocModel} in a unified scheme, then obtain a semi-discrete wave system
%\begin{eqnarray}\label{DisModel}
%\begin{split}
%&\partial_t^2 u_{\vk}(t) +\A u_{\vk}(t) = f_{\vk}(t),\quad \vk \in \Z^d, t>0, \\
%& u_{\vk}(0) = \varphi_{0,\vk},\quad \partial_t u_{\vk}(0) = \varphi_{1,\vk}, \quad \vk\in \Z^d,
%\end{split}
%\end{eqnarray}
%with
%\beq \label{A}
%\A u_{\vk}= \sum_{|\vk-\vm|\le L}a_{\vk,\vm}u(\vxm),
%\eeq
%where $L$ is a positive integer.

 %Let $\mathcal{T}_h$ be a uniform rectangular grid over the whole domain that is uniform in all directions with the same mesh size $h$. Define mesh points $\{\vxk=\vk h\}_{\vk\in\Z^d}$ where $\vk=k_1$ for 1D case and $\vk=(k_1,k_2)$ for 2D case

 In this section, we discretize the nonlocal operator \eqref{L} using the high-order quadrature-based FD scheme and approximate the temporal derivative using the explicit FD scheme to achieve a fully discrete wave system over the whole space.
\subsection{Discretization of the nonlocal operator}
%We use the asymptotically compatible (AC) scheme developed in \cite{du2019asymptotically} to approximate the nonlocal operator $\L_{\delta}$. The AC schemes, first proposed in \cite{tian2014asymptotically}, are numerical discretizations of nonlocal models that converge to nonlocal continuum models for a fixed horizon parameter and to the local discrete schemes as the horizon vanishes for both discrete schemes with a fixed numerical resolution and for continuum models with increasing numerical resolution.
%A second-order AC quadrature-based finite difference scheme for nonlocal operator is proposed in \citet{tian2013analysis} for 1-D problems and \citet{du2019asymptotically, du2018nonlocal}.
 {
Here we extend the second-order quadrature-based FD scheme approximating spatially nonlocal operators \citep[see][]{tian2013analysis,du2019asymptotically, du2018nonlocal} to arbitrarily high-order scheme.
 %And the scheme is asymptotically compatible (AC) \cite{tian2014asymptotically}, i.e., the numerical discretization of nonlocal model that converges to nonlocal continuum model for a fixed horizon parameter with increasing numerical resolution and converges to the local discrete schemes with a fixed numerical resolution as the horizon vanishes.
%for both discrete schemes with a fixed numerical resolution and for continuum models with increasing numerical resolution.
%offer the potential to solve for approximations of a model of interest with different choices of parameters to gain efficiency and to avoid the pitfall of reaching inconsistent limits.
 %For simplicity, we denote $u_{\vk}$ as approximation of $u(\vxk)$.
First we state some norm notations used in the whole paper. The notations $\vert\cdot\vert_1$ and $\|\cdot\|$ stand for the $\ell_1$ norm and standard Euclidean norm (i.e., $\ell^2$-norm) in the $d$-dimensional vector space, resprectively. And $|\cdot|_\infty$ represents the maximum norm, concretely, for a vector $\vx\in\mathbb{R}^d$ or a matrix $A\in \mathbb{R}^{m\times n}$,
$$
|\vx|_\infty = \max_{i={1,\cdots,d}}|x_i|, \quad |A|_\infty = \max_{i={1,\cdots,m};j={1,\cdots,n}}|A_{i,j}|.
$$
 Let $\{\bm{x_k} =\vk h\}_{\vk\in\Z^d}$ be the set of nodes (grid points) of the uniform rectangular grid $\mathcal{T}_h$ over the whole space with mesh size $h$, where $\vk$ denotes a multiindex.
The nonlocal operator \eqref{L} acting on $u(\bm{x_k})$ can be written as
\beq \label{nonop}
\L_{\delta} u(\bm{x_k}) = \int_{B_{\delta}(\bm{x_k})} \frac{u(\bm{x_k})-u(\bm{y})}{w(\bm{x_k}-\vy)} w(\bm{x_k}-\vy)\gamma(\bm{x_k}-\vy)d\vy,
%&=& \int_{\R^d} \frac{u(\bm{x_k})-u(\bm{x_k}-\bm{s})}{w(\bm{s})} w(\bm{s})\gamma(|\bm{s}|)d\bm{s}.
\eeq
% {where the integral is limited to $B_{\delta}(\bm{x_k})$ (an interval ($d=1$) or square ($d=2$) with length $2\delta$, centered at $\bm{x_k}$) according to the condition \eqref{ass2} on kernel.}
where the weight function $w(\vz) = {\|\vz\|^2}/{|\vz|_1}$
is introduced to ensure the approximate scheme is asymptotically compatible (AC) \citep[see][]{du2019asymptotically}, a concept proposed by \citet{tian2014asymptotically}, which means that the solution of the scheme converges to that of the corresponding local continuum models when both horizon $\delta$ and mesh size $h$ tend to zero, regardless of how $\delta$ and $h$ may or may not be dependent (see \citet{tian2020asymptotically}, \citet{tian2013analysis}, \citet{tian2014asymptotically} for further information). The property of AC is vital in multiscale modelling and computation.
%However, the weight function $w(\vz)$ can not make the high-order discretizations proposed next AC.
In this work, we focus on the case of fixed $\delta$, so whether the numerical scheme is AC is not our main concern.

We use the idea of composite integration rule to compute the integral \eqref{nonop}. First we divide the integral domain $B_{\delta}(x_{\vk})$ into $(2L/p)^d$ equal small domains $T^{\vk}_{\vi}$, where $L=\lceil \delta/h \rceil$. For an example, the 1D domain $T^{k}_{i} $ is given as
\beqs
T^{k}_{i} = [x_{k}-Lh+(i-1)ph,~ x_{k}-Lh+iph], \quad i=1,2,\dots,2L/p.
\eeqs
For simplicity, we always choose $h$ that can make $\delta$ an integral multiple of $h$ and $L$ an integral multiple of $p$. Then on each small domain $T^{\vk}_{\vi}$, we use the $p$th-degree Lagrange interpolation to approximate the integrand part $\displaystyle \frac{u(\bm{x_k})-u(\bm{y})}{w(\bm{x_k}-\vy)}$. The rest part $w(\bm{x_k}-\vy)\gamma(\bm{x_k}-\vy)$ can be regarded as the integral weight. Let $u_{\vk}$ be the approximation of $u(\bm{x_k})$, $\Phi_{\vk,p}(\vx)$ be the $p$th-degree (1D) or bi$p$th-degree (2D) Lagrange polynomial at point $\bm{x_k}$ on each divided small domain $T^{\vk}_{\vi}$, then one obtains the discretization for \eqref{nonop} as
 }
 %Using the $p$th-degree Lagrange interpolation to approximate $\displaystyle \frac{u(\bm{x_k})-u(\bm{y})}{w(\bm{x_k}-\vy)}$ in \eqref{nonop} with respect to the variable $\vy$, one obtains
% By interpolating $\displaystyle \frac{u(\bm{x_k})-u(\bm{y})}{w(\bm{x_k}-\vy)}$ in \eqref{nonop} with respect to the variable $\vy$, one can obtain the following approximation to nonlocal operator
%\beq \label{DisNon1}
%\L_{\delta,h} u_{\vk} = \int_{\R^d} \I_{h,p}\left( \frac{u_{\vk}-u(\bm{y})}{w(\bm{x_k}-\vy)} \right)w(\bm{x_k}-\vy)\gamma(|\bm{x_k}-\vy|)d\vy,
%\eeq
%where $\I_{h,p}$ represents the $p$th-degree Lagrange interpolation operator associated with the grid $\mathcal{T}_h$.
%%Let $\Phi_{\vk,p}(\vx)$ be the standard continuous piecewise $p$th-degree ($d=1$) or bi-$p$th-degree ($d=2$) basis function at point $\bm{x_k}$.
%Let $\Phi_{\vk,p}(\vx)$ be the standard continuous piecewise $p$th-degree basis function at point $\bm{x_k}$.
%After a simple calculation, one can rewrite \eqref{DisNon1} as
\beq\label{a}
\begin{aligned}
\L_{\delta,h} u_{\vk}
&= \sum_{\vm\in \Z^d,\vm\neq\vk} \frac{u_{\vk}-u_{\vm}}{w(\bm{x_k}-\bm{x_m})} \int_{B_{\delta}(\bm{x_k})} \Phi_{\vm,p}(\vy)w(\bm{x_k}-\vy)\gamma(\bm{x_k}-\vy)d\vy \\
&= \sum_{\vm\in \Z^d,\vm\neq\vk} \frac{u_{\vk}-u_{\vm}}{w(\bm{x_k}-\bm{x_m})} \int_{B_{\delta}(\bm{0})} \Phi_{\vm-\vk,p}(\vs)w(\vs)\gamma(\vs)d\vs \\
&= \sum_{\vm\in \Z^d}a_{\vk-\vm}(u_{\vk}-u_{\vm}),
%&=& \sum_{\vm\in \Z^d}a_{\vk-\vm}u_{\vm}, \label{a}
\end{aligned}
\eeq
where
%\begin{eqnarray*}
%a_{\vk-\vm} = \left\{
%\begin{array}{cc}
%\displaystyle \frac{1}{w(\vxk-\vxm)}\int_{\R^2} \Phi_{\vm}(\vy)w(\bm{x_k}-\vy)\gamma_0(\bm{x_k}-\vy)d\vy, & \vm \neq \vk,\\
% \displaystyle 0, & \vm = \vk.
%\end{array}
%\right.
%\end{eqnarray*}
\begin{eqnarray}\label{coe}
a_{\vm} = \left\{
\begin{array}{cc}
\displaystyle \frac{1}{w(h\vm)}\int_{B_{\delta}(\bm{0})} \Phi_{\vm,p}(\vs)w(\vs)\gamma(\vs)d\vs, & \vm\neq \bm{0},\\
 \displaystyle 0, & \vm = \bm{0}.
\end{array}
\right.
\end{eqnarray}
%It is obvious that
%\begin{eqnarray}\label{coe_a}
%a_{\vm} = \left\{
%\begin{array}{cc}
%-b_{\vm}, & \vm \neq \bm{0},\\
% \displaystyle  \sum_{\vm\in \Z^d,\vm\neq\bm{0}} b_{\vm}, & \vm = \bm{0}.
%\end{array}
%\right.
%\end{eqnarray}
%Both forms \eqref{b} and \eqref{a} will be used in the designing of ABCs.
%\begin{eqnarray}\label{a}
%a_{\vm} = \left\{
%\begin{array}{cc}
%\displaystyle -\frac{1}{w(h\vm)}\int_{\R^d} \Phi_{\vm,p}(\vy)w(\vy)\gamma(|\vy|)d\vy, & \vm \neq \bm{0},\\
% \displaystyle -\sum_{\vm\in \Z^d,\vm\neq\bm{0}} a_{\vm}, & \vm = \bm{0}.
%\end{array}
%\right.
%\end{eqnarray}
%and
%\begin{eqnarray*}
%a_{\vk,\vm} = \left\{
%\begin{array}{cc}
%\displaystyle -a_{\vk-\vm},& \vm \neq \vk, \\
% \displaystyle -\sum_{\vm\neq\vk} a_{\vk,\vm} , & \vm = \vk.
%\end{array}
%\right.
%\end{eqnarray*}
It is obvious that $a_{\vm}=a_{-\vm}$.
%For a fixed $\delta>0$, let $L=\lceil{\delta}/{h}\rceil$. According to the condition \eqref{ass2}, the coefficients $\{a_{\vm}\}$ satisfy
%\beq
%&& a_{\vm} = 0,\quad |\vm|>L.
%\eeq
%According to the theory of the composite integration, one can verify that the coefficients $a_{\vm}\ge0$ for the order of interpolation $p\le7$ \citep[see][]{stoer2013introduction}.

%We point out that the finite difference scheme for the Laplacian also satisfies the form \eqref{a} with different $a_{\vm}$.
%For further study, we can equivalently reformulate $\L_{\delta,h}$ in \eqref{a} into
According to the finite horizon assumption \eqref{ass2} of kernel, the coefficient $a_{\vm}$ satisfies
$$
a_{\vm} = 0, \quad |\vm|_\infty\ge L.
$$
For further study,
the following equivalent form of the $\L_{\delta,h}$ is needed
%Based on the \eqref{a} and the property $a_{\vm}=a_{-\vm}$, one can rewrite operator $\L_{\delta,h}$ in an equivalent form
\beq
\L_{\delta,h} u_{\vk} = \sum_{|\vm|_{\infty}\le L}c_{\vm}u_{\vk+\vm},\quad \vk\in\Z^d
\eeq
with
\beq
c_{\vm}=
\left\{
\begin{array}{cc}
-a_{\vm},& \vm\neq \bm{0},\\
\displaystyle \sum_{\vm\in \Z^d,\vm\neq\bm{0}} a_{\vm},&\vm=\bm{0},
\end{array}
\right.
\eeq
where the property $a_{\vm}=a_{-\vm}$ is used.

 {
On the truncation error of quadrature-based FD approximation \eqref{a}, we have the following lemma.
\begin{lemma}\label{trun_err}
%（与核函数有关）$G(\vs;\vxk)$的$p+1$次导有界时，
If $u\in C_b^{p+3}(\R^d)$ and $w(\vs)\gamma(\vs)$ is integrable in $B_\delta(\bm{0})$,
%then the approximation error of scheme \eqref{a} is $\O(h^q)$, i.e.,
then it holds that
\beq
|\L_{\delta,h} u -\L_\delta u|_\infty \le Ch^q,
\eeq
where $C$ is a constant independent of $h$. And the order $q$ is given as
\beq
q=
\left\{
\begin{array}{ll}
p+1,& p~ \text{is odd},\\
p+2,& p~ \text{is even and}~ u\in C_b^{p+4}(\R^d).
\end{array}
\right.
\eeq
%When $p$ is odd, $q=p+1$; when $p$ is even, moreover if $u\in C_b^{p+4}(\R^d)$, then $q=p+2$.
\end{lemma}
For brevity, the proof of this lemma is given in Appendix.
\begin{remark}
We discuss the symbol of the coefficient $a$ given in \eqref{coe}.
For the case of $p=1$, all coefficients $a_{\vm}$ are non-negative since the basis function $\Phi_{\vm,1}(\vs)$ is non-negative. When $p\ge2$, the situation is complicated. The sign of $a_{\vm}$ depends on the kernel, and it cannot be guaranteed to be always non-negative.
%For the symbol of coefficient $a$, through the direct calculation. Let $\delta/h=Cp$ (C is some positive integer) when meshing, then we can easily calculate the coefficients $a_{\vm}$ under some given kernel functions.
Through the direct calculation, one has that the coefficients are non-negative when $p\le 6$ for the constant kernel. And for the common used kernels $\gamma_1(\vs)=C\|\vs\|^{-1}$ and $\gamma_2(\vs)=C\|\vs\|^{-2}$, the corresponding coefficients are non-negative when $p\le 7$ and $p\le 3$, respectively. %For general kernel functions, the sign of $a_{\vm}$ is complicated.
\end{remark}
\begin{remark}
The weight function $w(z)$ is introduced to guarantee the numerical approximation with the linear interpolation (i.e., $p=1$) is asymptotically compatible in \citet{du2019asymptotically}. Here we keep the weight $w(z)$ in discretization since the introduction of $w(z)$ also can relax the requirement for kernel function \citep[see][]{du2019asymptotically}.
%while the high-order discretization with the weight function $w(z)$ may not be AC.
\end{remark}
}

\subsection{Fully discrete wave system}
 Let $\mathcal{T}_\tau = \{t_n | t_n=n\tau; \; 0\leq n\leq N \}$ be a uniform
partition of $[0,T]$ with the time step size $\tau=T/N$, and $u_{\vk}^{(n)}$ be the approximation of $u(x_{\vk},t_n)$. Define the second-order approximation for time derivative by
\beq
\D_{\tau} u^{(n)}= \frac{1}{\tau^2}\big(u^{(n+1)}-2u^{(n)}+u^{(n-1)}\big).
\eeq
Using the explicit finite difference method to discretize the temporal derivative for problem \eqref{NonModel},
we have the fully discrete system on the whole space as
\begin{align}\label{DisModel}
&\D_{\tau} u^{(n)}_{\vk} + \L_{\delta,h} u^{(n)}_{\vk} = f^{(n)}_{\vk},&& \vk \in\Z^d, n\ge1, \\
&u^{(0)}_{\vk} = \varphi_{\vk}, && \vk\in\Z^d, \label{Initial1}\\
&u_{\vk}^{(1)}=\varphi_{\vk}+\tau \psi_{\vk}+\frac{\tau^2}{2}(-\L_{\delta,h}\varphi_{\vk}+f_{\vk}^{(0)}),&& \vk\in\Z^d.\label{Initial2}
\end{align}

\section{Design of absorbing boundary conditions}
We now consider the construction of DtN-type ABCs for the fully discrete system \eqref{DisModel}-\eqref{Initial2} based on the DtD-type mappings proposed by \citet{du2018nonlocal, du2018numerical}.
%We apply the discrete Green's first identity to derive the DtN mapping based on the DtD mapping.
We first streamline the useful notations and  tools. Let $\Omega=\{\vx\in\R^d:|\vx|_{\infty}< \beta\}$ be the computational domain of interest, where $\beta$ is a positive real number. Set $M=\lceil{\beta}/{h}\rceil$.
%Let $t_n=n\tau$ be the time grid points with time mesh size $\tau > 0$, and denote the approximation of $u(\vxk,t_n)$ by $u_{\vk}^{(n)}$. Apply the explicit difference scheme in time, we obtain the following fully discrete system of the semi-discrete problem \eqref{DisModel}:
%\begin{eqnarray}\label{DisModel}
%\begin{split}
%&\D_{\tau}u_{\vk}^{(n)} +\A u_{\vk}^{(n)} = f_{\vk}^{(n)},\quad \vk \in \Z^d, n\ge1, \\
%& u_{\vk}^0 = \varphi_{0,\vk},\quad u_{\vk}^1 = \varphi_{1,\vk}, \quad \vk\in \Z^d,
%\end{split}
%\end{eqnarray}
%where $\A$ is defined in \eqref{A} and $\D_{\tau}$ is the second order central difference operator given by
%$$\D_{\tau} u^{(n)} = \frac{u^{(n+1)}-2u^{(n)}+u^{(n-1)}}{\tau^2}.$$
To clearly address the index in various grid domains, we define
\begin{eqnarray}\label{sym}
\begin{split}
& \k = \{\vk\in \Z^d:|\vk|_{\infty}< M\}, \quad \k^c=\{\vk\in \Z^d: |\vk|_{\infty}\ge M\}, \\
&\k^-=\{\vk\in \Z^d: |\vk|_{\infty}< M-L\},\quad \k_{\gamma}^-=\{\vk\in \Z^d: M-L\le |\vk|_{\infty}< M\},\\
&\k^+= \{\vk\in \Z^d: |\vk|_{\infty}< M+L\} ,\quad\k_{\gamma}^+=\{\vk\in \Z^d: M\le |\vk|_{\infty}< M+L\}.
\end{split}
\end{eqnarray}
We also introduce the $z$-transform and its inverse transform for a bounded infinite sequence $\{u^{(n)}\}_{n=0}^{+\infty}$ as
\begin{align}\label{z-tran}
&\hat{u}(z)=\sum_{n=0}^{+\infty}z^{-n}u^{{(n)}}, \quad |z|>1,\\
&u^{(n)}=\frac{1}{2\pi {\rm i}}\int_{C_{\rho}}\hat{u}(z)z^{n-1}dz,\quad n\ge0,\quad \rho>1, \label{Ivzf}
\end{align}
where $z$ is a continuous complex variable, $C_{\rho}$ is a counterclockwise circle with a radius of $\rho$.

For vectors $\bm{u}=\{u_{\vk}\}_{\vk\in\F}$ and $\bm{v}=\{v_{\vk}\}_{\vk\in\F}$ ($\F$ indicates any subset of $\Z^d$, such as $\k, \k^c$),
%the inner product is defined by
%$$\left(\bm{u},\bm{v} \right)_{\F}=\sum_{\vk\in\F}u_{\vk} v_{\vk}.$$
%The corresponding norm is defined as
%$$\|\bm{u}\|_{\F}^2=(\bm{u},\bm{u})_\F.$$
 the $\ell^2$-inner product and norm are respectively given as
\beq \label{norm1}
\left(\bm{u},\bm{v} \right)_{\F}=\sum_{\vk\in\F}u_{\vk} v_{\vk} \quad \text{and}\quad \|\bm{u}\|_{\F}=\sqrt{(\bm{u},\bm{u})_\F}.
\eeq
And we denote $L^2$-inner product by $(\cdot, \cdot)$, i.e.,
%In general, the inner product in $L^2(\Omega)$ space is defined as
\begin{eqnarray*}
(f,g)=\int_{\Omega}f(\vx)g(\vx)d\vx,\quad \forall f,g \in L^2(\Omega).
\end{eqnarray*}
The discrete $L^2$-inner product and norm in $\F\subset\Z^d$ are defined as
\begin{eqnarray*}
\left(\bm{u},\bm{v} \right)_{h,\F}=h^d\sum_{\vk\in\F}u_{\vk} v_{\vk}, \quad \quad \|\bm{u}\|_{h,\F}=\sqrt{(\bm{u},\bm{u})_{h,\F}}.
\end{eqnarray*}
%In particular, we introduce a discrete nonlocal seminorm as
%\begin{eqnarray}\label{nonlocalnorm}
% |\bm{u} |_{h,\F}  =\frac{h^d}{2}\sum_{\vk\in\F}\sum_{\vm\in\F}a_{\vk-\vm}(u_{\vk}-u_{\vm})^2.
%\end{eqnarray}
%One can verify that if the coefficients $a_{\vk-\vm}$ given in \eqref{coe} is obtained from the second-order central difference approximation of the Laplacian, the seminorm $ |\bm{u} |_{h,\F}$ is exactly the classical discrete $H^1$-seminorm.
And we define a discrete bilinear form $\left<\cdot,\cdot\right>_{h,\F}$ by
\begin{eqnarray}\label{nonlocalnorm}
\left<\bm{u},\bm{v}\right>_{h,\F} =\frac{h^d}{2}\sum_{\vk\in\F}\sum_{\vm\in\F}a_{\vk-\vm}(u_{\vk}-u_{\vm})(v_{\vk}-v_{\vm}).
\end{eqnarray}
Then $ |\bm{u} |_{h,\F}:=\sqrt{\left<\bm{u},\bm{v}\right>_{h,\F}}$ is a discrete seminorm.
%One can verify that if the coefficients $a_{\vk-\vm}$ given in \eqref{coe} is obtained from the second-order central difference approximation of the Laplacian, the seminorm $ |\bm{u} |_{h,\F}$ is exactly the classical discrete $H^1$-seminorm.
%the seminorm induced by it
%\begin{eqnarray} \label{seminorm}
% |\bm{u} |_{h,\F} =  \sqrt{\left<\bm{u},\bm{u} \right>_{h,\F}}.
% \end{eqnarray}
For brevity, for any vector confined on the index set $\k$, we omit the subscripts $\k$ in the notation below, such as
% and simply write $\left(\bm{u},\bm{v} \right)_h$ for $ \left(\bm{u},\bm{v} \right)_{h,\k}$, $\|\bm{u}\|_h$ for $\|\bm{u}\|_{h,\k}$,etc.
\begin{eqnarray*}
&& \left(\bm{u},\bm{v} \right)_{h,\k} :=\left(\bm{u},\bm{v} \right)_h, \quad \quad \|\bm{u}\|_{h,\k} :=\|\bm{u}\|_h.
\end{eqnarray*}
%Specially, we try to introduce a discrete semi-norm for nonlocal model.

%Then we define an inner product as follows
%\begin{align}
%& \left<\bm{u},\bm{v} \right>_{h,\F} = h^d\sum_{\vk\in\F}\sum_{\vm\in\F}a_{\vk-\vm}(u_{\vk}-u_{\vm})v_{\vk} \nonumber\\
%& \quad \quad \quad  \quad  =\frac{h^d}{2}\sum_{\vk\in\F}\sum_{\vm\in\F}a_{\vk-\vm}(u_{\vk}-u_{\vm})(v_{\vk}-v_{\vm}), \label{semi}
%\end{align}
%and semi norm induced by it is
%\begin{eqnarray} \label{Semi-norm}
% |\bm{u} |_{h,\F}^2 =  \left<\bm{u},\bm{u} \right>_{h,\F}.
% \end{eqnarray}

%For the local problem, we remark that the norm $ |\bm{u} |_{h,\F}^2$ is the discrete $H^1$ semi norm.
%$$|\bm{u} |_{1,\F}^2=\sum_{\vk\in\F} (\frac{u_{\vk}-}{})$$.

\subsection{DtD-type absorbing boundary conditions}
To construct the DtN-type ABCs, we briefly review the design of DtD-type ABCs.
As the initial data $\varphi$, $\psi$ and the source function $f$ are compactly supported, we assume that
\beq
f(\vx,t)= \varphi(\vx)= \psi(\vx)=0,\quad |\vx|_{\infty}>\beta.
\eeq
%We first divide $\Z^d$ into two parts: the bounded part $\k$ and the unbounded part $\k^c$.
The problem \eqref{DisModel}-\eqref{Initial2} is equivalent to the following two subproblems. The first subproblem is defined on the index set $\k$ as
\beq
\begin{aligned}\label{InModel}
&\D_{\tau}u^{(n)}_{\vk} + \L_{\delta,h}u^{(n)}_{\vk}=f^{(n)}_{\vk},&& \vk \in\k, n\ge1, \\
&u^{(0)}_{\vk} = \varphi_{\vk}, && \vk\in\k, \\
&u_{\vk}^{(1)}=\varphi_{\vk}+\tau \psi_{\vk}+\frac{\tau^2}{2}(-\L_{\delta,h}\varphi_{\vk}+f_{\vk}^{(0)}),&& \vk\in\k.
\end{aligned}
\eeq
The second subproblem is defined on the index set $\k^c$ as
\begin{eqnarray} \label{ExtModel}
\begin{aligned}
&\D_{\tau}u^{(n)}_{\vk} +  \L_{\delta,h}u^{(n)}_{\vk}=0,&& \vk \in\k^c, n\ge1, \\
& u_{\vk}^{(0)} = 0,\quad u_{\vk}^{(1)} = 0, &&\vk \in\k^c.
\end{aligned}
\end{eqnarray}
Problems \eqref{InModel} and \eqref{ExtModel} can not be solved independently, since they are related through the boundary.
Following the idea presented in \citet{du2018nonlocal, du2018numerical}, the value of $\{u_{\vk}\}_{\vk\in\k_{\gamma}^+}$ can be expressed by $\{u_{\vk}\}_{\vk\in\k_{\gamma}^-}$ through considering the exterior problem \eqref{ExtModel}. One may apply the $z$-transform to \eqref{ExtModel} to have
\beq \begin{split} \label{zExtModel}
&s\hat{u}_{\vk}+\sum_{|\vm|_{\infty}\le L}c_{\vm}\hat{u}_{\vk+\vm}=0, \quad \vk\in\k^c, \\
& \lim_{|\vk|\rightarrow+\infty} \hat{u}_{\vk}=0,
\end{split} \eeq
where $\displaystyle s=(z^{-1}-2+z)/\tau^2$.

To investigate the well-posedness of problem \eqref{zExtModel}, we introduce the sequence space equipped with the $\ell^2$-norm
$$
\ell^2=\{\hat{\vu}=\{\hat{u}_{\vk}\}_{\vk\in\Z^d}:\|\hat{\vu}\|^2=\sum_{\vk\in\Z^d}|\hat{u}_{\vk}|^2<+\infty\},
$$
%and its dense linear subspace
%$$l_0=\{\vu=\{u_{\vk}\}_{\vk\in\Z^d}\in l^2:\sharp \vu<+\infty \}.$$
%In the above, the symbol $\sharp$ represents the number of nonzero elements.
and define the linear operator $\L_{\delta,h}$ on $\ell^2$ as
\begin{equation*}
\L_{\delta,h} \hat{\vu}=\left\{ \sum_{|\vm|_\infty\le L} c_{\vm}\hat{u}_{\vk+\vm}\right\}_{\vk\in\Z^d},\quad
\forall~ \hat{\vu}=\{\hat{u}_{\vk}\}_{\vk\in\Z^d}\in \ell^2.
\end{equation*}
One can verify that $\L_{\delta,h}$ is nonnegative and symmetric.
Therefore, the spectrum set $\sigma(\L_{\delta,h})$ is located in the positive-half real axis.
% and for all $s\in\mathbb{C}$ with $\Re(s)>0$, the discrete problem \eqref{zExtModel} is well-posed.
%Based on the above discussion, one has that the discrete problem \eqref{zExtModel} is well-posed for all $s\notin\sigma(\mathcal{-\L}_{\delta,h})$ if providing a boundary condition on the left of the domain $\k^{r,c}$.

For all $s\notin\sigma(\mathcal{-\L}_{\delta,h})$ and prescribed the boundary data $\hat{u}_{\vk}$ with all $\vk\in\k^-_{\gamma}$,  the exterior problem \eqref{zExtModel} admits a unique solution.
Accordingly, we expect that the values of $u_{\vk}$ on $\k_{\gamma}^+$ can be expressed by the values on $\k_{\gamma}^-$, i.e., there exists a matrix function with entries
%In particular, we expect that there exists a matrix function
$$\hat{\K}_{\vk,\vm}=\hat{\K}_{\vk,\vm}(z),\quad \vk\in\k_{\gamma}^+, ~\vm\in\k_{\gamma}^-, $$
such that
\beq \label{DtD}
\hat{u}_{\vk}=\sum_{\vm\in\k_{\gamma}^-}\hat{\K}_{\vk,\vm}\hat{u}_{\vm}, \quad \vk\in\k_{\gamma}^+.
\eeq
%This is a DtD-type mapping.
Applying the inverse $z$-transform to \eqref{DtD}, one has the following DtD-type ABC:
\beq  \label{DtD0}
u_{\vk}^{(n)}= \sum_{\vm\in \k_{\gamma}^-}(\K_{\vk,\vm}*u_{\vm})^{(n)} =\sum_{\vm\in \k_{\gamma}^-} \sum_{j=0}^{n} \K_{\vk,\vm}^{(n-j)}u_{\vm}^{(j)}, \quad \vk\in\k_{\gamma}^+,
\eeq
where
\beq \label{int2d}
\K_{\vk,\vm}^{(j)} = \frac{1}{2\pi {\rm i}}\int_{C_\rho}\hat{\K}_{\vk,\vm}(z)z^{j-1}dz,\quad \rho>1,~ j\ge0.
\eeq
%Once we can obtain the formula of $\hat{\K}_{\vk,\vm}$,
%However, it is generally hard to express $\hat{\K}_{\vk,\vm}$ in an analytic form.

In simulations, we utilize the trapezoidal rule to approximate the contour integral in \eqref{int2d}, i.e.,
 %in operator $\K^{(j)}$ by
\begin{equation} \label{discreteCoeApprox}
\K_{\vk,\vm}^{(j)}\approx \widetilde{\K}_{\vk,\vm}^{(j)}= \frac{\rho^{j}}{P}\sum_{p=1}^{P}\hat{\K}_{\vk,\vm}(\rho e^{2\pi {\rm i} p/P})e^{2\pi {\rm i} jp/P },\quad  j \ge0,
\end{equation}
where $P$ is a positive integer. And for any $\varepsilon>0$, one can take $P$ large enough such that
\begin{equation}\label{K_err}
|\K^{(j)} - \widetilde{\K}^{(j)} |_\infty \leq \varepsilon, \quad \forall j.
\end{equation}

So far, we achieve a discrete initial-boundary-value problem with the  DtD-type ABCs
\beq
\begin{aligned}
&\D_{\tau}u_{\vk}^{(n)} +\L_{\delta,h} u_{\vk}^{(n)} = f_{\vk}^{(n)},&& \vk \in\k, n\ge1, \\
& u^{(n)}_{\vk} =  \sum_{\vm\in \k_{\gamma}^-} \sum_{j=0}^{n} \widetilde{\K}_{\vk,\vm}^{(n-j)}u_{\vm}^{(j)}, && \vk\in\k^+_{\gamma}, n\ge1,\\
&u^{(0)}_{\vk} = \varphi_{\vk}, && \vk\in\k, \\
&u_{\vk}^{(1)}=\varphi_{\vk}+\tau \psi_{\vk}+\frac{\tau^2}{2}(-\L_{\delta,h}\varphi_{\vk}+f_{\vk}^{(0)}),&& \vk\in\k.
\end{aligned}
\eeq

In the following, we use the recently developed methods in \citet{du2018nonlocal, du2018numerical} to address how to achieve the formula of $\hat{\K}_{\vk,\vm}$ for the 1D and 2D cases, respectively.

%\begin{remark}
%In practical simulations, we utilize the trapezoidal rule to approximate the contour integral in \eqref{int2d}, i.e.,
% %in operator $\K^{(j)}$ by
%\begin{equation} \label{discreteCoeApprox}
%\K_{\vk,\vm}^{(j)}\approx \widetilde{\K}_{\vk,\vm}^{(j)}= \frac{\rho^{-j}}{P}\sum_{p=1}^{P}\hat{\K}_{\vk,\vm}(e^{2\pi {\rm i} p/P})e^{-2\pi {\rm i} jp/P },\quad  j >0,
%\end{equation}
%where $P$ is a positive integer. And for any $\varepsilon>0$, one can take $P$ large enough such that
%\begin{equation}\label{K_err}
%\|\K^{(j)} - \widetilde{\K}^{(j)} \|_\infty \leq \varepsilon.
%\end{equation}
%\end{remark}

\bigskip
\noindent{\bf One-dimensional case.} For the 1D case, $\k = (-M,M)\cap\Z$, then we divide $\k^c$ into two index subsets $\k^{c,r}=\{k\in\Z:k\ge M\}$ and $\k^{c,l}=\{k\in\Z:k\le-M\}$. Similarly, let $\k_{\gamma}^{+,r}=\{{M},\dots,{M+L-1}\}$, $\k_{\gamma}^{+,l}=\{{-M-L+1},\dots,{-M}\}$, then $\k_{\gamma}^+=\k_{\gamma}^{+,r}\cup\k_{\gamma}^{+,l}$.

We first consider the right exterior problem, i.e., the discrete problem \eqref{zExtModel} restricted to $\k^{c,r}$.
Let us introduce a family of vectors as
$$
\hat{\vU}_{M,q}=[\hat{u}_{M+(q-1)L},\dots,\hat{u}_{M+qL-1}]^T,~ q=0,1,\dots.$$
Then, the discrete problem \eqref{zExtModel} restricted on $\k^{c,r}$ can be rewritten as
\begin{eqnarray} \label{matrix}
&&s\hat{\vU}_{M,q}+A\hat{\vU}_{M,q-1}+B\hat{\vU}_{M,q}+A^T\hat{\vU}_{M,q+1}=\bm{0},\quad q\ge 1,\\
&& \lim_{q\rightarrow +\infty}\hat{\vU}_{M,q}= \bm{0}, \label{matrix1}
\end{eqnarray}
where the coefficient matrices $A$ and $B$ are given as
\begin{equation} \label{AB}
A=\begin{pmatrix}
c_L&\cdots&\cdots&c_2&c_1\\
&c_L&\cdots&\cdots&c_2\\
&&c_L&\cdots&\cdots\\
&&&\cdots&\cdots\\
   &   &   &   &c_L
\end{pmatrix},\ B=\begin{pmatrix}
c_0&c_1&\cdots&\cdots&c_{L-1}\\
c_1&c_0&c_1&\cdots&\cdots\\
\cdots&c_1&c_0&c_1&\cdots\\
\cdots&\cdots&\cdots&\cdots&\cdots\\
c_{L-1}&\cdots&\cdots&c_1&c_0
\end{pmatrix}.
\end{equation}
Set $A_0=-\left(s+B\right)^{-1}A$ and $B_0=-\left(s+B\right)^{-1}A^T$. Eq. \eqref{matrix} can be further written as
\begin{equation}\label{matrix2}
\begin{split}
\hat{\vU}_{M,q}=A_0\hat{\vU}_{M,q-1}+B_0\hat{\vU}_{M,q+1},\quad q\ge 1.
\end{split}
\end{equation}

Prescribed $\hat{\vU}_{M,0}$, from \eqref{matrix2} with boundary condition \eqref{matrix1}, one can express $\hat{\vU}_{M,1}$ by $\hat{\vU}_{M,0}$  as
\begin{equation}\label{matrix operator}
\hat{\vU}_{M,1}=\hat{\K}_r(s)\hat{\vU}_{M,0}.
\end{equation}
Specifically for $L=1$, both $A_0$ and $B_0$ degenerate to scalars. So the mapping $\hat{\K}_r(s)$ can be computed analytically as
%\begin{equation}
%\hat{\K}_r(s)=\frac{2+h^2s-\sqrt{4h^2s+h^4s^2}}{2}.
%\end{equation}
\begin{equation}
\hat{\K}_r(s)=\frac{c_0+s-\sqrt{2c_0s+s^2}}{c_0}.
\end{equation}

However, for the case of $L\ge2$, it is nontrivial to find the exact expression of $\hat{\K}_r(s)$. We use the iterative technique proposed in \citet{du2018nonlocal} to numerically calculate $\hat{\K}_r(s)$ to have
\begin{equation} \label{kernelK}
\hat{\K}_r(s)=A_0+B_0[A_1+B_1[\dots+B_{m-1}[A_m+B_m[\dots]]]],
\end{equation}
where $A_0$ and $B_0$ are given in \eqref{AB}, and $A_m$, $B_m$ ($m\geq 1$) are computed iteratively by
\begin{equation} \label{AmBm}
\small
%\begin{pmatrix}
%A_{m}&B_{m}
%\end{pmatrix}
(A_{m}\;\; B_{m})=
%\begin{pmatrix}
%0&I&0
%\end{pmatrix}
(0\;\; I\;\;0)
\begin{pmatrix}
I&-B_{m-1}&0\\
-A_{m-1}&I&-B_{m-1}\\
0&-A_{m-1}&I
\end{pmatrix}^{-1}
\begin{pmatrix}
A_{m-1}&0\\0&0\\0&B_{m-1}
\end{pmatrix}.
\end{equation}

%Since the condition at infinity \eqref{matrix1},  $\hat{\K}_r(s)$ can be calculated by truncating the series terms in \eqref{kernelK}.  For a given tolerance error $\epsilon:= 10^{-14}$, we stop evaluating $A_{m+1}$ and $B_{m+1}$ when the infinite norm of $A_m$ and $B_m$ in \eqref{AmBm} is less than $\epsilon$. And in numerical experiments, the maximum number of the iterations to make $\hat{\K}_r(s)$ converge is generally less than 20 steps.

%Then for any prescribed $s\notin\sigma(-\L_{\delta,h})$, using the inverse $z$-transform to \eqref{matrix operator} and combining the convolution theorem, one obtains the right DtD-type mapping
%\begin{equation}\label{con}
%\bm{U}_{M,1}^{(n)}=\K_r^{(n)}*\bm{U}_{M,0}^{(n)}=\sum_{j=0}^{n} \K_r^{(n-j)}\bm{U}_{M,0}^{(j)},\quad n\ge 0,
%\end{equation}
%where $*$ stands for the discrete convolution operator and
%\begin{equation}\label{int1d}
%\K_r^{(j)}=\frac{1}{2\pi {\rm i}}\int_{|z|=\rho}\hat{\K}_r(s)z^{j-1}dz, \quad \rho>1,\quad j>0.
%\end{equation}

%One can see that the expression \eqref{con} is a DtD-type mapping on the right artificial boundary points.

By analogy with the design of DtD-type mapping on the right, one can derive a DtD-type mapping on the left as
\begin{equation}\label{con1}
\hat{\vU}_{-M,1}=\hat{\K}_l(s)\hat{\vU}_{-M,0},
\end{equation}
where
$$\hat{\vU}_{-M,q}=[\hat{u}_{-M-(q-1)L},\dots,\hat{u}_{-M-qL+1}]^T,\quad q=0,1.$$

%\rd{We remark that the mappings \eqref{con} and \eqref{con1} can be rewritten as the pointwise form \eqref{DtD}.}

 {
\begin{remark}
By truncating the series terms in the formula \eqref{kernelK}, we obtain the approximation of the operator $\hat{\K}$. The truncation criterion is to introduce a tolerance error, which is set as $\epsilon:= 10^{-14}$, such that the $L^2$-norms of $A_m$ and $B_m$ in \eqref{AmBm} are less than the given tolerance $\epsilon$. This provides an efficient way of evaluating $\hat{\K}$ in \eqref{kernelK} for the problem considered in this paper. It turns out that the maximum number of the iteration to obtain the converged $\hat{\K}$ for the given $\epsilon$ is less than 20 in all simulations.
\end{remark}
}

\bigskip
\noindent{\bf Two-dimensional case.} We utilize the methodology of the nonlocal potential theory to design the DtD-type ABCs for two-dimensional case \citep[see][]{du2018numerical}.
%To understand the philosophy of nonlocal potential theory, we
Let $G_{\vk}=G_{\vk}(z)$ be the fundamental solution of the equation \eqref{zExtModel} with $s(z)\not\in\sigma(-\L_{\delta,h})$, this is, $G_{\vk}$ satisfies the governing equation
\begin{eqnarray}
\label{greeneqn}
&&sG_{\vk}+\sum_{|\vm|_\infty\le L}c_{\vm} G_{\vk+\vm}=\delta_{\vk,\mathbf{0}},\ \vk\in \Z^2,\\
\label{infinitybc2}
&& \lim_{ |\vk|\rightarrow+\infty}G_{\vk}= 0,
\end{eqnarray}
where $\delta_{\vk,\mathbf{0}}$ stands for the Kronecker symbol. The two-dimensional discrete Fourier transform (DFT) of $\{G_{\vk}\}_{\vk\in\Z^2}$ is defined as
$$
(\mathcal{F}G)_{\vxi}=\sum_{\vk\in \Z^2}G_{\vk} e^{-{\rm i}\vk\cdot\vxi},\ \vxi\in\R^2.
$$
By performing the DFT to \eqref{greeneqn}, one has
\beq \label{fft}
(\mathcal{F}G)_{\vxi}=\left(s+\sum_{|\vm|_{\infty}\le L}e^{{\rm i}\vm\bm\cdot\vxi}c_{\vm}\right)^{-1}, \ \vxi\in\R^2.
\eeq
Then using the inverse Fourier transform on $\{(\mathcal{F}G)_{\vxi}\}_{\vxi\in \R^2}$ yields
$$
G_{\vk}=\frac{1}{4\pi^2}\int_{[0,2\pi]^2}\left(s+\sum_{|\vm|_\infty\le L}e^{{\rm i}\vm\bm\cdot\vxi}c_{\vm}\right)^{-1}e^{{\rm i}\vk\cdot\vxi}d\vxi,\ \vk\in \Z^2.
$$
%which can be efficiently calculated with the FFT algorithm.

%Resorting to the idea of a first-kind integral equation. First we assume that the boundary of the equation \eqref{zExtModel} is
%$$\widetilde{u}_{\vk} = q_{\vk},~ \vk\in\k_{\gamma}^-.$$
%Then we have
%\beq \label{DtD1}
%\widetilde{u}_{\vk} = \sum_{\vm\in\k_{\gamma}^-}G_{\vk+\vm}q_{\vm},\quad |\vk|>M-L.
%\eeq
%Confined to the artificial boundary layer, we have
%\beq
%\hat{u}_{\vk} = \sum_{\vm\in\k_{\gamma}^-}G_{\vk+\vm}q_{\vm},\quad \vk\in \k_{\gamma}^-.
%\eeq

Following the idea of the potential theory, we assume the solution of \eqref{zExtModel} can be expressed as
%First we determine the potentials $q_{\vk}$ for all $\vk\in\k_{\gamma}^-$, such that the following solution expression is valid
\begin{eqnarray}\label{DtD1}
\hat{u}_{\vk}=\sum_{\vm\in\k_{\gamma}^-}G_{\vk+\vm}q_{\vm},\quad |\vk|_{\infty}>M-L,
\end{eqnarray}
where $q_{\vm}$ is the potential to be determined.
Confining \eqref{DtD1} to the boundary layer $\k_{\gamma}^-$ produces
$$\hat{u}_{\vk}=\sum_{\vm\in\k_{\gamma}^-}G_{\vk+\vm}q_{\vm},\quad \vk\in\k_{\gamma}^-.$$
%This is the discrete analogue of the first-kind integral equation.
Denote $(G_{\vk,\vm}^{-1})$ by the inverse matrix of the matrix $G_{\vk+\vm}$ with $\vk,\vm\in\k_{\gamma}^-$. Thus, the potential $q_{\vm}$ can be expressed by the fundamental solution $G_{\vk}$ and the value of $\hat{u}_{\vk}$ on $\vk\in\k_{\gamma}^-$ in the form of
\beq \label{PS}
q_{\vm} = \sum_{\vl\in\k_{\gamma}^-}G_{\vm,\vl}^{-1}\hat{u}_{\vl},\quad \vm\in \k_{\gamma}^-.
\eeq
Substituting \eqref{PS} into \eqref{DtD1} and restricting to the boundary $\k_{\gamma}^+$ yields
\beq \label{DtD2}
\hat{u}_{\vk} = \sum_{\vm\in \k_{\gamma}^-}\hat{\K}_{\vk,\vm}\hat{u}_{\vm}, \quad \vk\in\k_{\gamma}^+,
\eeq
with
$$\hat{\K}_{\vk,\vm}=\sum_{\vl\in\k_{\gamma}^-}G_{\vk+\vl}G^{-1}_{\vl,\vm},\quad \vk\in\k_{\gamma}^+,\ \vm\in\k_{\gamma}^-.$$

 {
We point out that our procedure of adopting the potential theory to construct ABCs for the two-dimensional discrete system is similar to the difference potential method (DPM) proposed by Ryabenkii et al. \citep[see,][and references therein]{ryaben2006a,tsynkov1996artificial}. The DPM is also based on the potential theory, and concretely, which needs to formulate an appropriate auxiliary problem (to simplify the numerical implementation), and then construct the boundary equation with projection. The DPM has been successfully applied to design ABCs at irregular artificial boundaries for local problems. %The approach needs to formulate an appropriate auxiliary problem (to simplify the numerical implementation), and then construct the boundary equation with projection. The essence of the DPM is the difference analogy to this approach. Here we formulate the discrete boundary integral equation and then directly solve it.
%Since there is no complete potential theory for the nonlocal problem, we just imitate the classical potential method of the local problem to assume \eqref{DtD1} holds on. Then we directly solve the discrete boundary integral equation.   %generalized potential, and construct an auxiliary problem to simplify the computation.
%no grid adaptation to the shape of $\Gamma$ is required
}

\subsection{DtN-type absorbing boundary conditions}
%The study of Neumann problems plays an important role on many applications such as interface problems, free boundary problems and so on.
%And the DtN-type ABCs also facilitate the stability and convergence analysis.
In order to construct the DtN-type ABC based on the DtD-type mapping \eqref{DtD0}, we now introduce the formula of nonlocal Neumann boundary.
%nonlocal Neumann boundary condition \cite{du2012analysis,du2019uniform}
By analogy with the classical Green's first identity
$$
%(-\Delta\vu,\vv)=\int_{\Omega}\nabla \vu \cdot \nabla\vv-\int_{\partial\Omega}\frac{\partial \vu}{\partial \bm{n}}\vv dS.
(-\Delta u,v)_{\Omega}=(\nabla u, \nabla v)_{\Omega}-\left<\partial_{\vn} u, v\right>_{\partial\Omega},
$$
the nonlocal Green's first identity is given as:
\begin{align}
(\L_{\delta}u, v)
=&\int_{\vx\in\Omega}\int_{\vy\in\R^d}\left(u(\vx)-u(\vy)\right)v(\vx)\gamma(|\vx-\vy|)d\vy d\vx \nonumber\\
=&\frac12\int_{\vx\in\Omega}\int_{\vy\in\Omega}(u(\vx)-u(\vy))(v(\vx)-v(\vy))\gamma(|\vx-\vy|)d\vy d\vx \nonumber\\
& +\int_{\vx\in\Omega}\int_{\vy\in\Omega^c}(u(\vx)-u(\vy))v(\vx)\gamma(|\vx-\vy|)d\vy d\vx.
\label{NonGF}\end{align}
From \eqref{NonGF}, we have the nonlocal Neumann boundary (see the details in \citet{du2012analysis,du2019uniform})
\beq
\N u(\vx) = -\int_{\vy\in\Omega_{\gamma}^+}(u(\vx)-u(\vy))\gamma(|\vx-\vy|)d\vy, \quad \vx\in\Omega_{\gamma}^-,
\eeq
where $\Omega_{\gamma}^-=\{\vx\in\Omega:\text{dist}(\vx,\partial\Omega)\le\delta\}$ and the finite horizon property \eqref{ass2} is used to truncate the interaction domain. To obtain the formula of discrete Neumann boundary,  we perform the discrete nonlocal Green's first identity as
%\begin{equation} \label{green}
\begin{align} \label{green}
 (\L_{\delta,h} \bm{u},\bm{v})_{h} &= h^d\sum_{\vk\in \k}\L_{\delta,h} u_{\vk}\cdot v_{\vk} \notag\\
&=h^d \sum_{\vk\in \k}\sum_{\vm\in\k}a_{\vk-\vm}(u_{\vk}-u_{\vm})v_{\vk}+h^d\sum_{\vk\in\k}\sum_{\vm\in\k^c}a_{\vk-\vm}(u_{\vk}-u_{\vm})v_{\vk}  \notag \\
&=\frac{h^d}{2}\sum_{\vk\in\k}\sum_{\vm\in\k}a_{\vk-\vm}(u_{\vk}-u_{\vm})(v_{\vk}-v_{\vm}) +h^d\sum_{\vk\in\k^{-}_\gamma}\sum_{\vm\in\k^+_{\gamma}}a_{\vk-\vm}(u_{\vk}-u_{\vm})v_{\vk}  \notag\\
&= \left<\bm{u},\bm{v}\right>_h - \left(\N_\k \bm{u}, \bm{v}\right)_{\k^-_\gamma}.
\end{align}
%\end{equation}
%We denote by $\N_\k\vu$ the discrete nonlocal Neumann boundary,
In the above, the symbol $\left<\cdot, \cdot\right>_h$ and $\left(\cdot, \cdot\right)_{\k^-_\gamma}$ are defined in \eqref{nonlocalnorm} and \eqref{norm1}, respectively. And the discrete nonlocal Neumann boundary, denoted by $\N_\k\vu$, is formulated as
%where the operator $\N_{\k}\cdot$ is the discrete nonlocal Neumann operator
\begin{eqnarray}\label{NeuData}
&&\N_\k u_{\vk}=-h^d\sum_{\vm\in\k^+_\gamma}a_{\vk-\vm}(u_{\vk}-u_{\vm}),\quad \vk\in\k^-_\gamma.
%&&\N_{\k^c} U_k =\sum_{m\in\k}b_{k-m}(U_k-U_m),\quad k\in\k^c. \label{Neumc}
\end{eqnarray}

%\rd{And we introduce a new symbol $\left<\cdot,\cdot\right>_h$, which induces a seminorm as
%\begin{eqnarray} \label{seminorm}
% |\bm{u} |_{h,\F} =  \sqrt{\left<\bm{u},\bm{u} \right>_{h,\F}}.
% \end{eqnarray}
% One can verify that if the coefficients $a_{\vk-\vm}$ in \eqref{green} is obtained from the second-order central difference approximation of the Laplacian, the seminorm $ |\bm{u} |_{h,\F}$ is exactly the classical discrete $H^1$-seminorm.}

%Now we derive the DtN-type mapping from the DtD-type mapping \eqref{DtD}. We rewrite \eqref{NeuData} into
%\beq
%\N_\k u_{\vk}=\sum_{\vm\in\k^+_\gamma}a_{\vk,\vm}u_{\vk}-\sum_{\vm\in\k^+_\gamma}a_{\vk,\vm}u_{\vm},\quad \vk\in\k^-_\gamma.
%\eeq
Thus, we can reformulate the DtD-type mapping \eqref{DtD0} into the following DtN-type mapping
%\begin{equation}
\begin{align}\label{DtN}
\N_\k u_{\vk}^{(n)}
=&-h^d\sum_{\vm\in\k^+_\gamma}a_{\vk-\vm}\left(u_{\vk}^{(n)}-\sum_{\vl\in \k_{\gamma}^-}\widetilde{\K}_{\vm,\vl}^{(n)}*u_{\vl}^{(n)}\right) \notag\\
:=& \V^{(n)}_{\vk} u_{\vk}^{(n)},\quad \vk\in\k^-_{\gamma}.
\end{align}
%\end{equation}
%Taking the above as ABCs,
Finally, we obtain a numerical scheme with the DtN-type ABCs for the nonlocal problem \eqref{NonModel}
\beq\label{FinalScheme}
\begin{aligned}
&\D_{\tau}u_{\vk}^{(n)} +\L_{\delta,h} u_{\vk}^{(n)} = f_{\vk}^{(n)},&& \vk \in\k, n\ge1, \\
%& \N_\k u^{(n)}_{\vk} = -h^d\sum_{\vm\in\k^+_\gamma}a_{\vk-\vm}\left(u_{\vk}^{(n)}-\sum_{\vl\in \k_{\gamma}^-}\K_{\vm,\vl}^{(n)}*u_{\vl}^{(n)}\right),&& \vk\in\k^-_{\gamma}, n\ge1,\\
& \N_\k u^{(n)}_{\vk} =  \V^{(n)}_{\vk} u_{\vk}^{(n)}, && \vk\in\k^-_{\gamma}, n\ge1,\\
&u^{(0)}_{\vk} = \varphi_{\vk}, && \vk\in\k, \\
&u_{\vk}^{(1)}=\varphi_{\vk}+\tau \psi_{\vk}+\frac{\tau^2}{2}(-\L_{\delta,h}\varphi_{\vk}+f_{\vk}^{(0)}),&& \vk\in\k.
\end{aligned}
\eeq

%\begin{remark}
%Specially, for the one-dimensional second-order central difference scheme, the Neumann boundary \eqref{NeuData} is
%\beqs
%\N_\k u_{M}=\frac{u_{M+1}-u_{M}}{h}, \quad \N_\k u_{-M}=\frac{u_{-M-1}-u_{-M}}{h}.
%\eeqs
%This is
%the well-known discrete Neumann boundary condition and the equation \eqref{green} is exactly the traditional discrete Green's first identity.
%\end{remark}

%\rd{
%\begin{remark}
%We define a discrete nonlocal seminorm from \eqref{green} as
%\begin{eqnarray} \label{seminorm}
% |\bm{u} |_{h,\F} =  \sqrt{\left<\bm{u},\bm{u} \right>_{h,\F}}.
% \end{eqnarray}
%One can verify that if the coefficients $a_{\vk-\vm}$ in \eqref{green} is obtained from the second-order central difference approximation of the Laplacian, then the seminorm $ |\bm{u} |_{h,\F}$ is exactly the classical discrete $H^1$-seminorm.
%\end{remark}
%}

\begin{remark}
%Actually, we can also define the exterior Neumann boundary if we call \eqref{NeuData} interior Neumann boundary for $\k$. From \eqref{green}, we further have
%%\begin{equation} \label{green2}
%\begin{align}\label{green2}
%(\L_{\delta,h} \bm{u},\bm{v})_{h}
%=&\frac{h^d}{2}\sum_{\vk\in\k}\sum_{\vm\in\k}a_{\vk-\vm}(u_{\vk}-u_{\vm})(v_{\vk}-v_{\vm}) +h^d\sum_{\vk\in\k^{-}_\gamma}\sum_{\vm\in\k^+_{\gamma}}a_{\vk-\vm}(u_{\vk}-u_{\vm})v_{\vk} \notag \\
%=&\left<\vu,\vv \right>_h
%+h^d\sum_{\vk\in\k^{-}_\gamma}\sum_{\vm\in\k^+_{\gamma}}a_{\vk-\vm}(u_{\vk}-u_{\vm})(v_{\vk} -v_{\vm})\notag \\
%&-h^d\sum_{\vk\in\k^+_{\gamma}}\sum_{\vm\in\k^{-}_\gamma}a_{\vk-\vm}(u_{\vk}-u_{\vm})v_{\vk}.
%%&=& \left<\bm{u},\bm{v}\right>_h - (\N_\k \bm{u}, \vk-\vm)_{\k^-_\gamma} + (\N_\k \bm{u}, \bm{v})_{\k^-_\gamma},
%\end{align}
%%\end{equation}
%From the above equation, the exterior Neumann boundary for $\k$ can be defined as
%\begin{eqnarray}\label{NeuData2}
%&&\N_{\k^c} u_{\vk}=-h^d\sum_{\vm\in\k^-_\gamma}a_{\vk-\vm}(u_{\vk}-u_{\vm}),\quad \vk\in\k^+_\gamma.
%%&&\N_{\k^c} U_k =\sum_{m\in\k}b_{k-m}(U_k-U_m),\quad k\in\k^c. \label{Neumc}
%\end{eqnarray}
%This exterior Neumann boundary will be used in the stability analysis of the numerical scheme in the next section. We note that the interior Neumann boundary for $\k$ is equal to the exterior Neumann boundary for $\k^c$.
%%Since homogeneous assumption \eqref{c}, We know that $c_{\vm}=a_{\vk,\vk+\vm}$, $a_{\vk-\vm} = c_{\vm}$
We can also derive the Neumann boundary \eqref{NeuData} through considering the discrete nonlocal Green's first identity on the exterior domain $\k^c$,
\begin{align}\label{green2}
(\L_{\delta,h} \bm{u},\bm{v})_{h,\k^c}
=&\frac{h^d}{2}\sum_{\vk\in\k^c}\sum_{\vm\in\k^c}a_{\vk-\vm}(u_{\vk}-u_{\vm})(v_{\vk}-v_{\vm}) +h^d\sum_{\vk\in\k^{+}_\gamma}\sum_{\vm\in\k_{\gamma}^-}a_{\vk-\vm}(u_{\vk}-u_{\vm})v_{\vk} \notag \\
=&\left<\vu,\vv \right>_{h,\k^c}
+h^d\sum_{\vk\in\k^{+}_\gamma}\sum_{\vm\in\k^-_{\gamma}}a_{\vk-\vm}(u_{\vk}-u_{\vm})(v_{\vk} -v_{\vm})+\left(\N_\k \bm{u}, \bm{v}\right)_{\k^-_\gamma}.
\end{align}
This formula serves to bridge the interior and exterior problems, which will be used in stability analysis of  the numerical scheme \eqref{FinalScheme} in next section.
\end{remark}

\section{Stability and convergence analysis}
We now consider the stability of the following discrete system
\begin{align}
&\D_{\tau}\phi_{\vk}^{(n)} +\L_{\delta,h} \phi_{\vk}^{(n)} = g_{\vk}^{(n)},&& \vk\in\k, n\ge1, \label{stability}\\
& \N_\k \phi_{\vk}^{(n)} = \V^{(n)}_{\vk} \phi_{\vk}^{(n)} + g_{b,\vk}^{(n)},&& \vk\in \k^-_{\gamma}, n\ge1,\label{stability2}\\
%& \N_\k \phi_{\vk}^{(n)} =  -h^d\sum_{\vm\in\k^+_\gamma}a_{\vk-\vm}(\phi_{\vk}^{(n)}-\sum_{\vl\in \k_{\gamma}^-}\K_{\vm,\vl}^{(n)}*\phi_{\vl}^{(n)}) + g_{b,\vk}^{(n)},&& \vk\in \k^-_{\gamma}, n\ge1,\label{stability2}\\
%& \phi_{\vk}^0=\phi_{\vk}^1=0,&& \vk\in\k.\label{stability3}
&\phi_{\vk}^{(0)}=\mu^{(0)}_{\vk}, ~ \phi_{\vk}^{(1)}=\mu_{\vk}^{(1)},&& \vk\in\k,\label{stability3}
\end{align}
where $\bm{\mu}=\{\mu_{\vk}\}_{\vk\in\k}$ are the initial values, $\bm{g}=\{g_{\vk}\}_{\vk\in\k}$ and $\bm{g_b}=\{g_{b,\vk}\}_{\vk\in\k_{\gamma}^-}$ are the interior and boundary perturbation terms, respectively.
%can be considered as the interior and boundary perturbation terms from the truncation errors, respectively.
%where $$\N_\k \bm{\phi}_{M,0}^{(n)} =(D-A^{T}\K*)\bm{\phi}_{M,0}^{(n)},~ \N_\k \bm{\phi}_{-M,1}^{(n)} =(D-A^{T}\K*)\bm{\phi}_{-M,1}^{(n)},\text{(in theory)}, $$
% $$\tilde{\N}_\k \bm{\phi}_{M,0}^{(n)} =(D-A^{T}\tilde{\K}*)\bm{\phi}_{M,0}^{(n)},~ \tilde{\N}_\k \bm{\phi}_{-M,1}^{(n)} =(D-A^{T}\tilde{\K}*)\bm{\phi}_{-M,1}^{(n)},\text{(in practice)}.$$

 Define the discrete energy norm
\begin{eqnarray}\label{energynorm}
\| \bm{\phi}^{(n)} \|^2_E = \|  \D_{\tau}^F\bm{\phi}^{(n-1)} \|_h^2+ \frac14 |\bm{\phi}^{(n)}+\bm{\phi}^{(n-1)}|_h^2, \quad n\ge1,
\end{eqnarray}
where the forward difference operator $\D_\tau^F$ is given as
$\D_\tau^F u^{(n)}= \frac{1}{\tau}\big(u^{(n+1)}-u^{(n)}\big).$

\subsection{Stability analysis}
 \begin{theorem}\label{Th1}
 Take $S=2((2L+1)^d-1)|\va|_\infty$, where $\va$ is the coefficient of the discrete nonlocal operator defined in \eqref{coe}. When $\va\ge0$, there exist positive constants $C$ and $\tau_0$ such that for $\tau\le \min \{ \tau_0,2/\sqrt{S}\}$,
 the solution of \eqref{stability}-\eqref{stability3} satisfies the following stability estimate for $l\ge 2$:
 \begin{eqnarray} \label{stabilityanalysislocal}
 \| \bm{\phi}^{(l)} \|^2_E \le \Vert \D_{\tau}^F\bm{\mu}^{(0)} \Vert_h^2+C\tau \sum_{n=1}^{l-1}\left(\|\bm{g}^{(n)} \|_h^2+h^{-d}\| \bm{g_b}^{(n)} \|^2\right).
 \end{eqnarray}
% where $S=4dL|\va|$ for the local problem and $S=2((2L+1)^d-1)|\va|$ for the nonlocal problem.
 \end{theorem}

\begin{proof}
Taking the $L^2$-inner product between $\eqref{stability}$ and $(\bm{\phi}^{(n+1)}-\bm{\phi}^{(n-1)})$ on $\k$ yields
%\begin{eqnarray}\label{Proof1}
%\begin{split}
\begin{align}\label{Proof1}
&\left(\D_{\tau}\bm{\bm{\phi}}^{(n)},\bm{\bm{\phi}}^{(n+1)}-\bm{\bm{\phi}}^{(n-1)}\right)_h +\left( \L_{\delta,h} \bm{\phi}^{(n)},\bm{\phi}^{(n+1)}-\bm{\phi}^{(n-1)} \right)_h
=\left(\bm{g}^{(n)},\bm{\phi}^{(n+1)}-\bm{\phi}^{(n-1)}\right)_h.
\end{align}
%\end{split}
% \end{eqnarray}
The first term in the above equation can be written as
 \begin{eqnarray*}
 \begin{split}
 (\D_{\tau}\bm{\phi}^{(n)},\bm{\phi}^{(n+1)}-\bm{\phi}^{(n-1)})_h &= \left(\D_{\tau}^F\bm{\phi}^{(n)}-\D_{\tau}^F\bm{\phi}^{(n-1)}, \D_{\tau}^F\bm{\phi}^{(n)}+\D_{\tau}^F\bm{\phi}^{(n-1)} \right)_h\\
 &= \Vert \D_{\tau}^F\bm{\phi}^{(n)} \Vert_h^2-\Vert \D_{\tau}^F\bm{\phi}^{(n-1)} \Vert_h^2.
 \end{split}
\end{eqnarray*}
Applying the discrete nonlocal Green's first identity \eqref{green} to the second term of \eqref{Proof1}, one has
\begin{eqnarray}\label{ThA}
\begin{split}
& \left( \L_{\delta,h} \bm{\phi}^{(n)},\bm{\phi}^{(n+1)}-\bm{\phi}^{(n-1)} \right)_h \nonumber\\
%&=& h\sum_{k\in\k} \sum_{m\in\k}a_{k,m}(\bm{\phi}_k^{(n)}-\bm{\phi}_m^{(n)})\left(\bm{\phi}_k^{(n+1)}-\bm{\phi}_k^{(n-1)} \right)-\left( \N_{\k}\bm{\phi}^{(n)},\bm{\phi}^{(n+1)}-\bm{\phi}^{(n-1)} \right)_{\k_\gamma^-}\\
=& \left< \bm{\phi}^{(n)}, \bm{\phi}^{(n+1)}-\bm{\phi}^{(n-1)} \right>_h -\left(  {\N_{\k}}\bm{\phi}^{(n)},\bm{\phi}^{(n+1)}-\bm{\phi}^{(n-1)} \right)_{\k_\gamma^-} \nonumber\\
=& \frac14\left( |\bm{\phi}^{(n+1)}+\bm{\phi}^{(n)}|_h^2-|\bm{\phi}^{(n)}+\bm{\phi}^{(n-1)}|_h^2-|\bm{\phi}^{(n+1)}-\bm{\phi}^{(n)}|_h^2+|\bm{\phi}^{(n)}-\bm{\phi}^{(n-1)}|_h^2 \right) \nonumber\\
&- \left(  {\bm{\V}^{(n)}}\bm{\phi}^{(n)},\bm{\phi}^{(n+1)}-\bm{\phi}^{(n-1)} \right)_{\k_\gamma^-}- \left( \bm{g_b}^{(n)},\bm{\phi}^{(n+1)}-\bm{\phi}^{(n-1)} \right)_{\k_\gamma^-}, \nonumber\\
%&&- \left( (D-A^{T}\K*)\bm{\phi}_{M,0}^{(n)}+\bm{G}_r,\bm{\phi}^{(n+1)}-\bm{\phi}^{(n-1)} \right)_{\k^r}- \left( (D-A^{T}\K*)\bm{\phi}_{-M,1}^{(n)}+\bm{G}_l,\bm{\phi}^{(n+1)}-\bm{\phi}^{(n-1)} \right)_{\k^{(l)}}
\end{split}
\end{eqnarray}
where  the fact is used in the last equality that
$$a(b-c)=\frac14 \Big((a+b)^2-(a+c)^2-(a-b)^2+(a-c)^2\Big).$$
Summing index $n$ from $1$ to $l-1$ in \eqref{Proof1} and combining with initial conditions \eqref{stability3}, one obtains
%\begin{eqnarray}\label{interior}
\begin{align}\label{interior}
& \Vert \D_{\tau}^F\bm{\phi}^{(l-1)} \Vert_h^2-\Vert \D_{\tau}^F\bm{\mu}^{(0)} \Vert_h^2 +\frac14\left( |\bm{\phi}^{(l)}+\bm{\phi}^{(l-1)}|_h^2-|\bm{\phi}^{(l)}-\bm{\phi}^{(l-1)}|_h^2 \right) - \sum_{n=1}^{l-1}\left(   {\bm{\V}^{(n)}}\bm{\phi}^{(n)},\bm{\phi}^{(n+1)}-\bm{\phi}^{(n-1)} \right)_{\k_\gamma^-} \notag \\
&=\sum_{n=1}^{l-1}\left( \bm{g}^{(n)}, \bm{\phi}^{(n+1)}-\bm{\phi}^{(n-1)}\right)_h +\sum_{n=1}^{l-1}\left( \bm{g_b}^{(n)},\bm{\phi}^{(n+1)}-\bm{\phi}^{(n-1)} \right)_{\k_\gamma^-}.
\end{align}
%\end{eqnarray}

To estimate the second term associated with the boundary on the left side of the above equation, we consider the following exterior problem
\begin{align}
&\D_{\tau}\tilde{\phi}_{\vk}^{(n)} +\L_{\delta,h} \tilde{\phi}_{\vk}^{(n)} = 0,&& {\vk}\in\k^{c}, n\ge1, \label{EP}\\
&   { \tilde{\phi}_{\vk}^{(n)} = \phi_{\vk}^{(n)}}, &&  \vk\in \k^-_{\gamma}, n\ge1,\\
%~\N_\k \varphi_{\vk}^{(n)} = \N_\k \phi_{\vk}^{(n)}, &&  \vk\in \k^-_{\gamma}, n\ge1,\\
%&&\N_\k \varphi_{\vk}^{(n)} = \N_\k \phi_{\vk}^{(n)}, \quad \vk\in \k^-_{\gamma}, n\ge1,\\
%&& \lim_{k\rightarrow +\infty} \phi_k^{(n)} = 0, \quad n\ge1,\\
 &\tilde{\phi}_{\vk}^{(0)}=\tilde{\phi}_{\vk}^{(1)}=0, && \vk\in\k^{c}.
\end{align}

Taking the $L^2$-inner product between \eqref{EP} and $(\bm{\tilde{\phi}}^{(n+1)}-\bm{\tilde{\phi}}^{(n-1)})$ on the domain $\k^{c}$, one has
\begin{eqnarray*}
(\D_{\tau}\bm{\bm{\tilde{\phi}}}^{(n)},\bm{\bm{\tilde{\phi}}}^{(n+1)}-\bm{\bm{\tilde{\phi}}}^{(n-1)})_{h,\k^{c}} +\left( \L_{\delta,h} \bm{\tilde{\phi}}^{(n)},\bm{\tilde{\phi}}^{(n+1)}-\bm{\tilde{\phi}}^{(n-1)} \right)_{h,\k^{c}}=0.
\end{eqnarray*}
Summing index $n$ from $1$ to $l-1$ and combining with initial conditions and \eqref{green2},
one has
%\begin{equation}\label{exterior}
%\begin{split}
\begin{align}\label{exterior}
& \| \D_{\tau}^F\bm{\tilde{\phi}}^{(l-1)} \|_{h,\k^{c}}^2+\frac14\left( |\bm{\tilde{\phi}}^{(l)}+\bm{\tilde{\phi}}^{(l-1)}|_{h,\k^{c}}^2-|\bm{\tilde{\phi}}^{(l)}-\bm{\tilde{\phi}}^{(l-1)}|_{h,\k^{c}}^2 \right) \notag\\
 &+ \frac{h^d}{4}\sum_{\vk\in\k_{\gamma}^{+}}\sum_{\vm\in\k_{\gamma}^{-}}a_{\vk-\vm}\left( \left( (\tilde{\phi}_{\vk}^{(l)}-\phi_{\vm}^{(l)})+(\tilde{\phi}_{\vk}^{(l-1)}-\phi_{\vm}^{(l-1)}) \right)^2 \right.- \left.\left( (\tilde{\phi}_{\vk}^{(l)}-\phi_{\vm}^{(l)})-(\tilde{\phi}_{\vk}^{(l-1)}-\phi_{\vm}^{(l-1)}) \right)^2 \right)  \notag \\
 =& -\sum_{n=1}^{l-1}\left(  {\bm{\V}^{(n)}}\bm{\phi}^{(n)},\bm{\phi}^{(n+1)}-\bm{\phi}^{(n-1)} \right)_{\k_{\gamma}^{-}}.
 \end{align}
%\end{split}
%\end{equation}
%Similarly, by considering the left exterior problem, we can get
%\begin{eqnarray}\label{exteriorL}
%&& \|  \D_{\tau}^F\bm{\varphi}^{(l-1)} \|_{h,\k^{l,c}}^2+\frac14\left( |\bm{\varphi}^{(l)}+\bm{\varphi}^{(l-1)}|_{h,\k^{l,c}}^2-|\bm{\varphi}^{(l)}-\bm{\varphi}^{(l-1)}|_{h,\k^{l,c}}^2 \right) \nonumber \\
% &+& \frac h4\sum_{k\in\k^{l,+}}\sum_{m\in\k^{l,-}}a_{k,m}\left( \left( (\varphi_k^{(l)}-\varphi_m^{(l)})+(\varphi_k^{(l-1)}-\varphi_m^{(l-1)}) \right)^2- \left( (\varphi_k^{(l)}-\varphi_m^{(l)})-(\varphi_k^{(l-1)}-\varphi_m^{(l-1)}) \right)^2 \right)  \nonumber\\
% &+& \sum_{n=1}^{l-1}\left( \N_{\k}\bm{\phi}^{(n)},\bm{\varphi}^{(n+1)}-\bm{\varphi}^{(n-1)} \right)_{h,\k^{l,-}}=0,
%\end{eqnarray}
Substituting the left-hand side of  \eqref{exterior} into \eqref{interior}, one obtains
%Combining \eqref{interior} and \eqref{exterior}, we can obtain
%\begin{equation}\label{add}
%\begin{split}
\begin{align}\label{add}
& \Vert  \D_{\tau}^F\bm{\phi}^{(l-1)} \Vert_h^2+\frac14\left( |\bm{\phi}^{(l)}+\bm{\phi}^{(l-1)}|_h^2-|\bm{\phi}^{(l)}-\bm{\phi}^{(l-1)}|_h^2 \right) \notag\\
&+ \|  \D_{\tau}^F\bm{\tilde{\phi}}^{(l-1)} \|_{h,\k^c}^2+\frac14\left( |\bm{\tilde{\phi}}^{(l)}+\bm{\tilde{\phi}}^{(l-1)}|_{h,\k^c}^2-|\bm{\tilde{\phi}}^{(l)}-\bm{\tilde{\phi}}^{(l-1)}|_{h,\k^c}^2 \right)  \notag\\
 &+ \frac {h^d}{4}\sum_{\vk\in\k^{+}_\gamma}\sum_{\vm\in\k^-_\gamma}a_{\vk-\vm}\left( \left( (\tilde{\phi}_{\vk}^{(l)}-\phi_{\vm}^{(l)})+(\tilde{\phi}_{\vk}^{(l-1)}-\phi_{\vm}^{(l-1)}) \right)^2 \right.\notag \\
 &- \left.\left( (\tilde{\phi}_{\vk}^{(l)}-\phi_{\vm}^{(l)})-(\tilde{\phi}_{\vk}^{(l-1)}-\phi_{\vm}^{(l-1)}) \right)^2 \right)\notag \\
=&\Vert \D_{\tau}^F\bm{\mu}^{(0)} \Vert_h^2 +\sum_{n=1}^{l-1}\left( \bm{g}^{(n)}, \bm{\phi}^{(n+1)}-\bm{\phi}^{(n-1)}\right)_h+\sum_{n=1}^{l-1}\left( \bm{g_b}^{(n)},\bm{\phi}^{(n+1)}-\bm{\phi}^{(n-1)} \right)_{\k_\gamma^-}.
\end{align}
%\end{split}
%\end{equation}
Note that
%\begin{equation}\label{1}
%\begin{split}
\begin{align}\label{1}
|\bm{\phi}^{(l)}-\bm{\phi}^{(l-1)}|_h^2 &= \frac{h^d}{2}\sum_{\vk\in\k}\sum_{\vm\in\k}a_{\vk-\vm}((\phi^{(l)}_{\vk}-\phi^{(l-1)}_{\vk})-(\phi^{(l)}_{\vm}-\phi^{(l-1)}_{\vm}))^2 \notag\\
%&\le&  \frac{h^d}{2}\sum_{\vk\in\k}\sum_{\vm\in\k}a_{\vk-\vm}((\bm{\phi}^{(l)}_{\vk}-\bm{\phi}^{(l-1)}_{\vk})-(\bm{\phi}^{(l)}_{\vm}-\bm{\phi}^{(l-1)}_{\vm}))^2 \\
&\le h^d\sum_{\vk\in\k}\sum_{\vm\in\k}a_{\vk-\vm}((\phi^{(l)}_{\vk}-\phi^{(l-1)}_{\vk})^2+(\phi^{(l)}_{\vm}-\phi^{(l-1)}_{\vm})^2) \notag \\
%&\le& 2h^d((2L+1)^2-1)|\vb|\sum_{\vk\in\k}(\bm{\phi}^{(l)}_{\vk}-\bm{\phi}^{(l-1)}_{\vk})^2\\
%&=& 2((2L+1)^2-1)|\vb| \Vert  \D_{\tau}^F\bm{\phi}^{(l-1)}\Vert_h^2,
&\le S h^d\sum_{\vk\in\k}(\phi^{(l)}_{\vk}-\phi^{(l-1)}_{\vk})^2 \notag \\
&= S\tau^2 \Vert  \D_{\tau}^F\bm{\phi}^{(l-1)}\Vert_h^2,
\end{align}
%\end{split}
%\end{equation}
%where $S=4dL|\va|$ for the local problem and
where $S=2((2L+1)^d-1)|\va|_\infty$ and the property of $\va\ge0$ is used.
Similarly, one also has
\begin{eqnarray}
%&&|\bm{\varphi}^{(l)}-\bm{\varphi}^{(l-1)}|_{h,\k^c}^2 \le 2((2L+1)^2-1)|\vb| \Vert  \D_{\tau}^F\bm{\varphi}^{(l-1)}\Vert_{h,\k^c}^2,\\
|\bm{\tilde{\phi}}^{(l)}-\bm{\tilde{\phi}}^{(l-1)}|_{h,\k^c}^2 \le S\tau^2\Vert  \D_{\tau}^F\bm{\tilde{\phi}}^{(l-1)}\Vert_{h,\k^c}^2 \label{2}
\end{eqnarray}
and
\begin{align}
&h^d\sum_{\vk\in\k^{+}_\gamma}\sum_{\vm\in\k^-_\gamma}a_{\vk-\vm} \left( (\tilde{\phi}_{\vk}^{(l)}-\phi_{\vm}^{(l)})-(\tilde{\phi}_{\vk}^{(l-1)}-\phi_{\vm}^{(l-1)}) \right)^2 \nonumber\\
%&\le& 2((2L+1)^2-1)|\vb| (\Vert  \D_{\tau}^F\bm{\phi}^{(l-1)}\Vert_h^2+\Vert  \D_{\tau}^F\bm{\tilde{\phi}}^{(l-1)}\Vert_{h,\k^c}^2),
\le& S \tau^2 (\Vert  \D_{\tau}^F\bm{\phi}^{(l-1)}\Vert_h^2+\Vert  \D_{\tau}^F\bm{\tilde{\phi}}^{(l-1)}\Vert_{h,\k^c}^2) \label{3}.
\end{align}
Plugging \eqref{1}, \eqref{2} and \eqref{3} into \eqref{add}, one yields
\begin{eqnarray*}
\begin{split}
& (1-S\tau^2/4)\Vert  \D_{\tau}^F\bm{\phi}^{(l-1)}\Vert_h^2+\frac14 |\bm{\phi}^{(l)}+\bm{\phi}^{(l-1)}|_h^2
+ (1-S\tau^2/4 )\| \D_{\tau}^F\bm{\tilde{\phi}}^{(l-1)} \|_{h,\k^c}^2 \nonumber \\
&+ \frac14 |\bm{\tilde{\phi}}^{(l)}+\bm{\tilde{\phi}}^{(l-1)}|_{\k^c}^2
 + \frac {h^d}{4}\sum_{\vk\in\k^{+}_\gamma}\sum_{\vm\in\k^-_\gamma}a_{\vk-\vm} \left( (\tilde{\phi}_{\vk}^{(l)}-\phi_{\vm}^{(l)})+(\tilde{\phi}_{\vk}^{(l-1)}-\phi_{\vm}^{(l-1)}) \right)^2 \nonumber\\
\le&\Vert \D_{\tau}^F\bm{\mu}^{(0)} \Vert_h^2+\sum_{n=1}^{l-1}\left( \bm{g}^{(n)}, \bm{\phi}^{(n+1)}-\bm{\phi}^{(n-1)}\right)_h+\sum_{n=1}^{l-1}\left( \bm{g_b}^{(n)},\bm{\phi}^{(n+1)}-\bm{\phi}^{(n-1)} \right)_{\k_\gamma^-}.
\end{split}
\end{eqnarray*}
%Define the discrete energy norm
%\begin{eqnarray}
%\| \bm{\phi}^{(n)} \|^2_E = \|  \D_{\tau}^F\bm{\phi}^{(n-1)} \|_h^2+ \frac14 |\bm{\phi}^{(n)}+\bm{\phi}^{(n-1)}|_h^2.
%\end{eqnarray}
If $(1-S\tau^2/4) > 0$, then there exists a positive constant $C$, s.t.,
\begin{eqnarray*}
\begin{split}
\| \bm{\phi}^{(l)} \|^2_E
\le& C\left(\Vert \D_{\tau}^F\bm{\phi}^{(0)} \Vert_h^2+ \sum_{n=1}^{l-1}\left( \bm{g}^{(n)}, \bm{\phi}^{(n+1)}-\bm{\phi}^{(n-1)}\right)_h\right.\notag\\
&+\left.\sum_{n=1}^{l-1}\left( \bm{g_b}^{(n)},\bm{\phi}^{(n+1)}-\bm{\phi}^{(n-1)} \right)_{\k_\gamma^-} \right) \\
=& C\Vert \D_{\tau}^F\bm{\phi}^{(0)} \Vert_h^2+C \tau \sum_{n=1}^{l-1}\left( \bm{g}^{(n)}, \D_{\tau}^F \bm{\phi}^{(n)}+\D_{\tau}^F\bm{\phi}^{(n-1)}\right)_h \\
&+
C \tau \sum_{n=1}^{l-1}\left( \bm{g_b}^{(n)}, \D_{\tau}^F \bm{\phi}^{(n)}+\D_{\tau}^F\bm{\phi}^{(n-1)}\right)_{\k_\gamma^-} \\
%&\le& C\tau \sum_{n=1}^{l-1}\| \bm{g}^{(n)} \|_h^2+C\tau \sum_{n=1}^{l-1}\| \bm{g_b}^{(n)} \|_h^2+C\tau \sum_{n=1}^{l-1}\|  \D_{\tau}^F\bm{\phi}^{(n)} \|_h^2.
%&\le& C\tau \sum_{n=1}^{l-1}\| \bm{g}^{(n)} \|_h^2+C\tau \sum_{n=1}^{l-1}\|  \D_{\tau}^F\bm{\phi}^{(n)} \|_h^2+C\tau h^{-d} \sum_{n=1}^{l-1}\| \bm{g_b}^{(n)} \|^2+ C\tau h^d\sum_{n=1}^{l-1}\| \D_{\tau}^F\bm{\phi}^{(n)} \|^2\\
\le& C\Vert \D_{\tau}^F\bm{\mu}^{(0)} \Vert_h^2+C\tau \sum_{n=1}^{l-1}\| \bm{g}^{(n)} \|_h^2+C\tau h^{-d}\sum_{n=1}^{l-1}\| \bm{g_b}^{(n)} \|^2 \\
&+C\tau \sum_{n=0}^{l-1}\|  \D_{\tau}^F\bm{\phi}^{(n)} \|_h^2.
% local
%&=& C \tau \sum_{n=1}^{l-1}\left( \bm{f}^{(n)}, \D_{\tau}^F \bm{\phi}^{(n)}+\D_{\tau}^F\bm{\phi}^{(n-1)}\right)_h \\
%&& + C\tau \sum_{n=1}^{l-1}\left( \D_{\tau}^F \left( \bm{g}^{(n-1)}, \bm{\phi}^{(n)}+\bm{\phi}^{(n-1)} \right)_{h,\k_\gamma^-}-\left(\D_{\tau}^F \bm{g}^{(n-1)}, \bm{\phi}^{(n)}+\bm{\phi}^{(n-1)} \right)_{h,\k_\gamma^-} \right)\\
%&\le& C\tau \sum_{n=1}^{l-1}\| \bm{f}^{(n)} \|_h^2 +C\tau \sum_{n=1}^{(l)}\| \D_{\tau}^F\bm{\phi}^{(n-1)} \|_h^2  +  C\tau \| \bm{g}^{(l-1)}\|_{h,\k_\gamma^-}^2 + C\tau \| \bm{\phi}^{(l)}+\bm{\phi}^{(l-1)} \|_{h,\k_\gamma^-}^2 \\
%&& + C\tau \sum_{n=1}^{l-1} \| \D_{\tau}^F \bm{g}^{(n-1)} \|_{h,\k_\gamma^-}^2 +C\tau \sum_{n=1}^{l-1}  \| \bm{\phi}^{(n)}+\bm{\phi}^{(n-1)} \|_{h,\k_\gamma^-}^2
\end{split}
\end{eqnarray*}
Applying the discrete Gronwall's inequality \citep[see][]{quarteroni1994numerical} for positive constant $\tau_0$ such that $\tau\le\tau_0$, one obtains \eqref{stabilityanalysislocal}.

\end{proof}

\subsection{Convergence analysis}
We now analyze the error of numerical scheme \eqref{FinalScheme} based on the above stability analysis. Let $\vu_{*}^{(n)}=\{u(x_{\vk},t_n)\}_{\vk\in\k}$ be
the vector whose entries are the nodal values of exact solutions of problem \eqref{NonModel} at the time $t_n$, and $\vu^{(n)}=\{u_{\vk}^{(n)}\}_{\vk\in\k}$ whose entries are the nodal values of solutions of the numerical scheme \eqref{FinalScheme}.
%Let $\vu^{(n)}=\{u_{\vk}^{(n)}\}_{\vk\in\k}$ and $\vu_{*}^{(n)}=\{u(x_{\vk},t_n)\}_{\vk\in\k}$ be the solution of the modified scheme \eqref{FinalScheme} and the nonlocal problem \eqref{NonModel}, respectively.
Denote the error by $\bm{\phi}^{(n)}=\vu_*^{(n)}-\vu^{(n)}$.
To perform the error estimate, we further introduce $\tilde{\vu}^{(n)}=\{\tilde{u}_{\vk}^{(n)}\}_{\vk\in\k}$ whose entries are the nodal values of numerical solutions of scheme \eqref{FinalScheme} with replacing the approximate $\widetilde{\K}$ by the exact $\K$. Then the error $\bm{\phi}^{(n)}$ can be divided into two parts, i.e.,
\begin{eqnarray*}
\bm{\phi}^{(n)}=(\vu_*^{(n)}-\tilde{\vu}^{(n)})+(\tilde{\vu}^{(n)}-\vu^{(n)}):=\bm{\phi}^{1,(n)}+\bm{\phi}^{2,(n)}.
\end{eqnarray*}

We now consider these two errors separately.
%Actually, the operator $\K$ in the operator $\V$ involved an inverse $z$-transform, which can not be computed exactly. In the practical simulation, we utilize the trapezoidal rule to approximate the operator $\K^{(j)}$ by
%\begin{equation} \label{discreteCoeApprox}
%\widetilde{\K}^{(j)}= \frac{\rho^{-j}}{P}\sum_{p=1}^{P}\hat{\mathcal{K}}(e^{2\pi i p/P})e^{-2\pi i jp/P },\quad j >0.
%\end{equation}
%And for any $\varepsilon>0$, we can take $P$ large enough such that
%\begin{equation}\label{K_err}
%\|\K^{(j)} - \widetilde{\K}^{(j)} \|_\infty \leq \varepsilon.
%\end{equation}
Note that the solution $\tilde{\vu}^{(n)}$
 is the same as the solution of fully discrete system \eqref{DisModel} confined on the computational domain since the discrete ABCs are exact.
%This implies that there is no error between the problem \eqref{FinalScheme} and \eqref{DisModel} on the computational domain.
Consequently, the error $\bm{\phi}^{1,(n)}$ only results from the approximation error of the fully discrete scheme \eqref{DisModel} to the original problem \eqref{NonModel}. Using the Taylor expansion, one immediately has
\beq
\|\bm{\phi}^{1,(n)}\|_E \le C_1(\tau^2+h^q),\quad 2\le n\le N,
\eeq
where $C_1$ is a positive constant, $q$ is given in Lemma \ref{trun_err}.

%We denote
%\beq
%\widetilde{\V}^{(n)}_{\vk}u_{\vk}^{(n)}=-h^d\sum_{\vm\in\k^+_\gamma}a_{\vk-\vm}(u_{\vk}^{(n)}-\sum_{\vl\in \k_{\gamma}^-}\widetilde{\K}^{(n)}_{\vm,\vl}*u_{\vl}^{(n)}),\quad \vk\in\k_{\gamma}^-.
%\eeq
%Therefore, the practical numerical scheme is the modified \eqref{FinalScheme} in which $\V^{(n)}_{\vk}$ is replaced by $\widetilde{\V}^{(n)}_{\vk}$.

On the other hand, one can verify $\bm{\phi}^{2,(n)}$ satisfies Eqs. \eqref{stability}-\eqref{stability3}
%with $\vg^{(n)}$ and $\bm{g_b}^{(n)}$ denoting the truncation errors of the modified scheme \eqref{FinalScheme} to the scheme \eqref{FinalScheme}, given by
with
$$\vg^{(n)} = \bm{0},$$
and
\begin{eqnarray} \label{ges}
g_{b,\vk}^{(n)} =h^d\sum_{\vm\in\k_{\gamma}^+}a_{\vk-\vm}\sum_{\vl\in\k_{\gamma}^-}(\K_{\vm,\vl}^{(n)}-\widetilde{\K}_{\vm,\vl}^{(n)})*u_{\vl}^{(n)},\quad \vk\in\k_{\gamma}^-.
\end{eqnarray}
According to the stability analysis in Theorem \ref{Th1}, we have
\beq
\| \bm{\phi}^{(l),2}\|_E \le \left(C\tau\sum_{n=1}^{l-1} h^{-d}\|\bm{g_b}^{(n)}\|^2 \right)^{\frac12},
\eeq
%Now we estimate the $\|\bm{g_b}^{(n)}\|$ given by
%%Furthermore, based on the above analysis and the definition of discrete DtN-type ABC \eqref{DtN}, we get
%\begin{align}\label{g2}
%&g_{b,\vk}^{(n)}=h^d\sum_{\vm\in\k_{\gamma}^+}a_{\vk-\vm}\sum_{\vl\in\k_{\gamma}^-}(\K_{\vm,\vl}^{(n)}-\widetilde{\K}_{\vm,\vl}^{(n)})*\tilde{u}_{\vl}^{(n)}, \quad \vk\in\k_{\gamma}^-.
%\end{align}
where $\|\bm{g_b}^{(n)}\|$ can be further estimated from \eqref{ges} and \eqref{K_err} by
%Combining \eqref{K_err}, we have the estimate
\begin{equation} \label{ges1}
 \|\bm{g_b}^{(n)}\| \leq  nh^{d}L^{2d+\frac12}|\bm{a}|_\infty |\tilde{\mathcal{K}}-\mathcal{K}|_\infty |\vu|_{[0,t_n]\times\k_{\gamma}^-}\le C nh^dL^{2d+\frac12}\varepsilon |\bm{a}|_\infty.
\end{equation}
Since the nonlocal horizon $\delta$ is fixed, we substitute $L=\delta/h$ into \eqref{ges1} to have
\begin{equation} \label{Ge}
 \|\bm{g_b}^{(n)}\| \le C n \delta^{2d+\frac12} h^{-d-\frac12}\varepsilon |\bm{a}|_\infty.
\end{equation}
The maximum norm of $\va$ depends on the kernels used in the nonlocal operator $\L_{\delta}$.
%For the constant kernels shown in the next numerical experiments
%$$\gamma(\bm{\alpha})=\frac{3}{d}\delta^{-2-d}, \quad |\bm{\alpha}|\le\delta,$$
%the scheme \eqref{coe} leads to $|\bm{a}| = \O(h^{d})$.
We here list three popularly and widely used kernel functions as
\begin{align}
%\begin{itemize}
%\item[(1)]
&\text{constant kernel:}\; \gamma(\bm{\alpha})=\frac{3}{d}\delta^{-2-d},~|\bm{\alpha}|_\infty\in [0,\delta]; \label{kernel1} \\
&\text{nonintegrable kernel:}\; \gamma(\bm{\alpha}) = 2\|\bm{\alpha}\|^{-1}\delta^{-2},~|\bm{\alpha}|_\infty\in (0,\delta];  \label{kernel2} \\
&\text{fractional Laplacian kernel:}\; \gamma(\bm{\alpha}) =  \frac{2^{2\nu}\nu \Gamma(\nu+d/2)}{\pi^{1/2}\Gamma(1-\nu)} \|\bm{\alpha}\|^{-d-2\nu}(0<\nu <1), ~|\bm{\alpha}|_\infty\in (0,\delta]. \label{kernel3}
%\end{itemize}
\end{align}
The scheme \eqref{coe} with any $p$ leads to $|\bm{a}|_\infty = \O(h^{d})$ for constant kernel \eqref{kernel1}. Similarly, one has $|\bm{a}|_\infty = \O(1)$ for kernel \eqref{kernel2} and $|\bm{a}|_\infty = \O(h^{-2\nu})$ for kernel \eqref{kernel3}.
Without loss of generality, we assume for convenience $|\bm{a}|_\infty=\O(h^{-r})$, where the index $r$ is determined by the kernel and the dimension of space.

To ensure that $\| \bm{\phi}^{(l),2}\|_E$ has the second-order accuracy in time, one can take $P$ large enough in \eqref{discreteCoeApprox} such that $\varepsilon  = \mathcal{O}(\tau^3h^{\frac{3d}{2}+\frac12+r})$.
 Then the total error $\bm{\phi}^{(n)}$ has the following estimate
\beq
\| \bm{\phi}^{(n)}\|_E
\le \| \bm{\phi}^{1,(n)}\|_E+\| \bm{\phi}^{2,(n)}\|_E \le C_1(\tau^2+h^p)+C_2\tau^2.
\eeq

Overall, we obtain the following error estimate of numerical scheme \eqref{FinalScheme}.
%\begin{theorem}\label{Error}
%Assume that the solution of the nonlocal wave equation $\eqref{NonModel}$ is sufficiently smooth, if $\tau\le \min \{ \tau_0,2/\sqrt{S}\}$ and taking $P$ large enough in \eqref{discreteCoeApprox} such that $\varepsilon  = \mathcal{O}(\tau^3h^{2v-d/2})$, then it holds
%\begin{eqnarray}
%\max_{0\leq n\leq N} \Vert \phi^{(n)} \Vert_h \leq C(\tau^2+h^p).
%\end{eqnarray}
%On the other hand, assume that the solution of the local wave equation $\eqref{LocModel}$ is sufficiently smooth, if $\tau\le \min \{ \tau_0,2/\sqrt{S}\}$ and taking $P$ large enough in \eqref{discreteCoeApprox} such that $\varepsilon  = \mathcal{O}(\tau^3h^{2D/2})$, then we have
%\begin{eqnarray}
%\max_{0\leq n\leq N} \Vert \phi^{(n)} \Vert_h \leq C(\tau^2+h^{2L}).
%\end{eqnarray}
%\end{theorem}
\begin{theorem}\label{Error}
Assume that the solution of the nonlocal wave equation $\eqref{NonModel}$ is sufficiently smooth. If $\tau\le \min \{ \tau_0,2/\sqrt{S}\}$ ($S$ is defined in Theorem \ref{Th1}) and taking $P$ large enough in \eqref{discreteCoeApprox} such that $\varepsilon =\mathcal{O}(\tau^3h^{\frac{3d}{2}+\frac12+r})$, then the following estimate holds
\begin{eqnarray}
\max_{2\leq n\leq N} \Vert \bm{\phi}^{(n)} \Vert_E \leq C(\tau^2+h^q),
\end{eqnarray}
where $q$ depends on the accuracy of the spatially discrete scheme.
\end{theorem}

 {
\begin{remark}
For fixed horizon parameter $\delta$, we now present a fine estimate on the time step size restriction given in Theorems \ref{Th1} and \ref{Error}, i.e., $\tau\le \min \{ \tau_0,2/\sqrt{S}\}$. Substituting $|\va|_\infty=\mathcal{O}(h^{-r})$ and $L=\delta/h$ into $S$, one has $\tau \le h^{(d+r)/2}$, which implies the time step restriction for different kernels as
\beq\label{time_restriction}
\tau\le
\left\{
\begin{array}{cc}
h^d,& \text{kernel \eqref{kernel1}},\\
h^{d/2},& \text{kernel \eqref{kernel2}}, \\
h^{d/2+\nu}, &\text{kernel \eqref{kernel3}}.
\end{array}
\right.
\eeq
\end{remark}
}

\section{Numerical Experiments}
We now provide two numerical examples to demonstrate the effectiveness of  our ABCs and the theoretical analysis. %In Example 1, we plot the evolution of the numerical solutions and investigate the convergence rate for the one-dimensional case.
Let $\bm{u}_{ref}$ and $\bm{u}_h$ be the solutions of problem \eqref{NonModel} and numerical scheme \eqref{FinalScheme}, respectively. The $L^2$-error and convergence rate are defined as
\begin{eqnarray}
&& L^2\text{-error}(h) = \|\bm{u}_h-\bm{u}_{ref}\|_h,\\
&& L^2\text{-rate} = \log\left(\frac{L^2\text{-error}(h_1)}{L^2\text{-error}(h_2)}\right) / \log\left(\frac{h_1}{h_2}\right).
\end{eqnarray}
%We compute the $L^2$-error and convergence rate of the numerical scheme \eqref{FinalScheme}.
%Similarly, the validity of ABCs and the convergence of numerical schemes for two-dimensional case are investigated in Example 2.

\begin{example}\label{ex1}
 We here consider 1D problem \eqref{NonModel} with $f(x,t)=0$. The initial values are given as
\beqs
&&\varphi(x)=\exp(-25(x-0.2)^2)+\exp(-25(x+0.2)^2),\\
&&\psi(x)=50x\exp(-25x^2).
\eeqs
We consider all three kernel functions \eqref{kernel1}-\eqref{kernel3} listed in section 4. For the convenience of exposition, we denote them by kernel-1, kernel-2 and kernel-3. And we choose $\nu=0.5$ in kernel-3.
%First, we consider the constant kernel \eqref{kernel1}.
In simulations, we set the computational domain $\Omega = (-2,2)$, the spatial mesh size $h=2^{-7}$, the time step size $\tau=2^{-8}$ {and the number of quadrature nodes given in \eqref{discreteCoeApprox} $P=20000$.} And the final time are $T=3,5,10$ for three kernels, respectively.
 Figure \ref{ex1sol} plots the evolutions of numerical solutions with the linear Lagrange interpolation when $\delta=0.25, 0.5$. One can see that the waves are effectively absorbed when they touch the boundaries, and no reflected wave is generated at boundaries to disrupt the solutions in the computational domain.
\begin{figure}[t!]
\centering
\begin{minipage}[]{\textwidth}
	\includegraphics[width=0.45\textwidth]{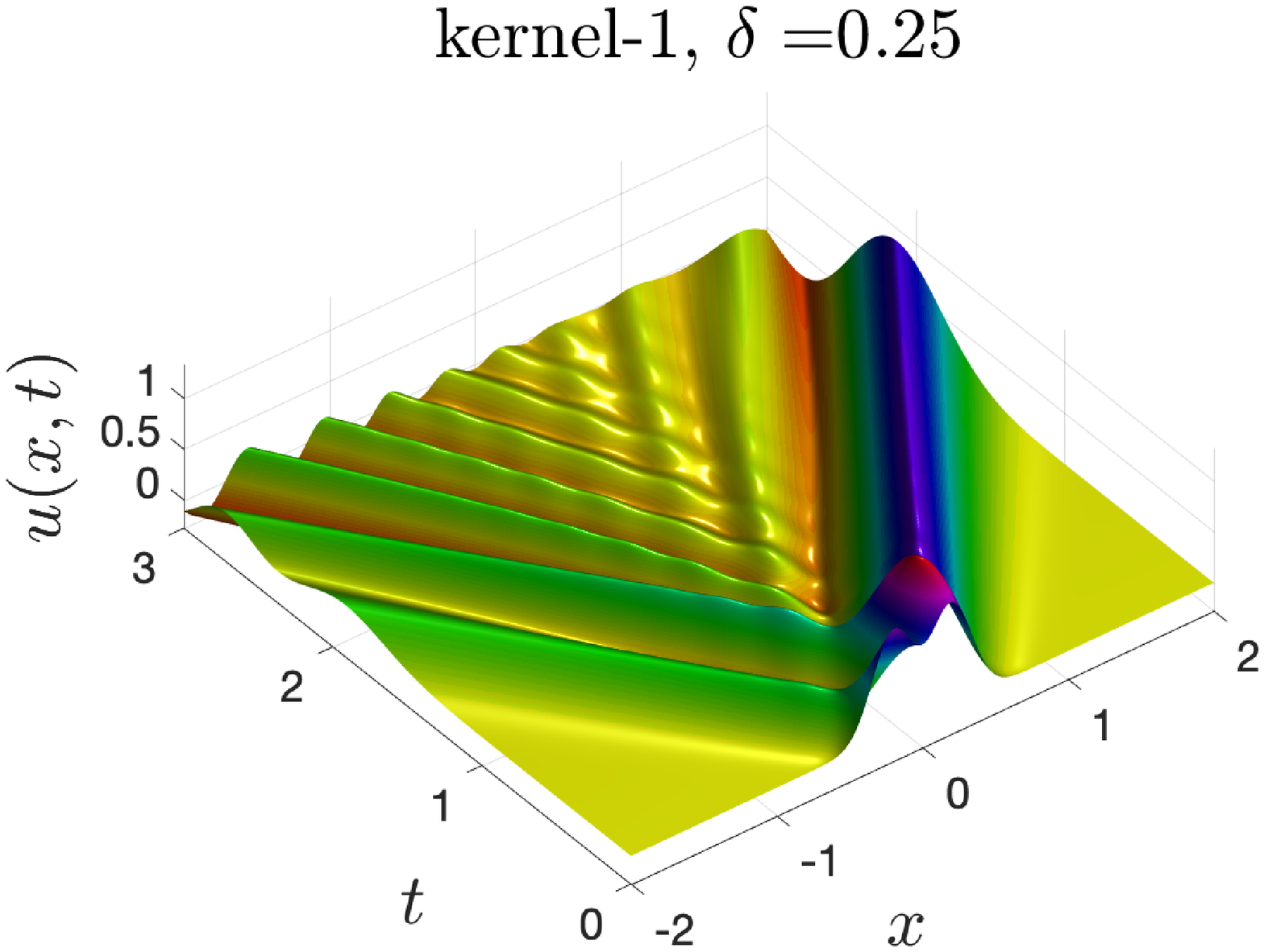}\hspace{0.2in}
	\includegraphics[width=0.45\textwidth]{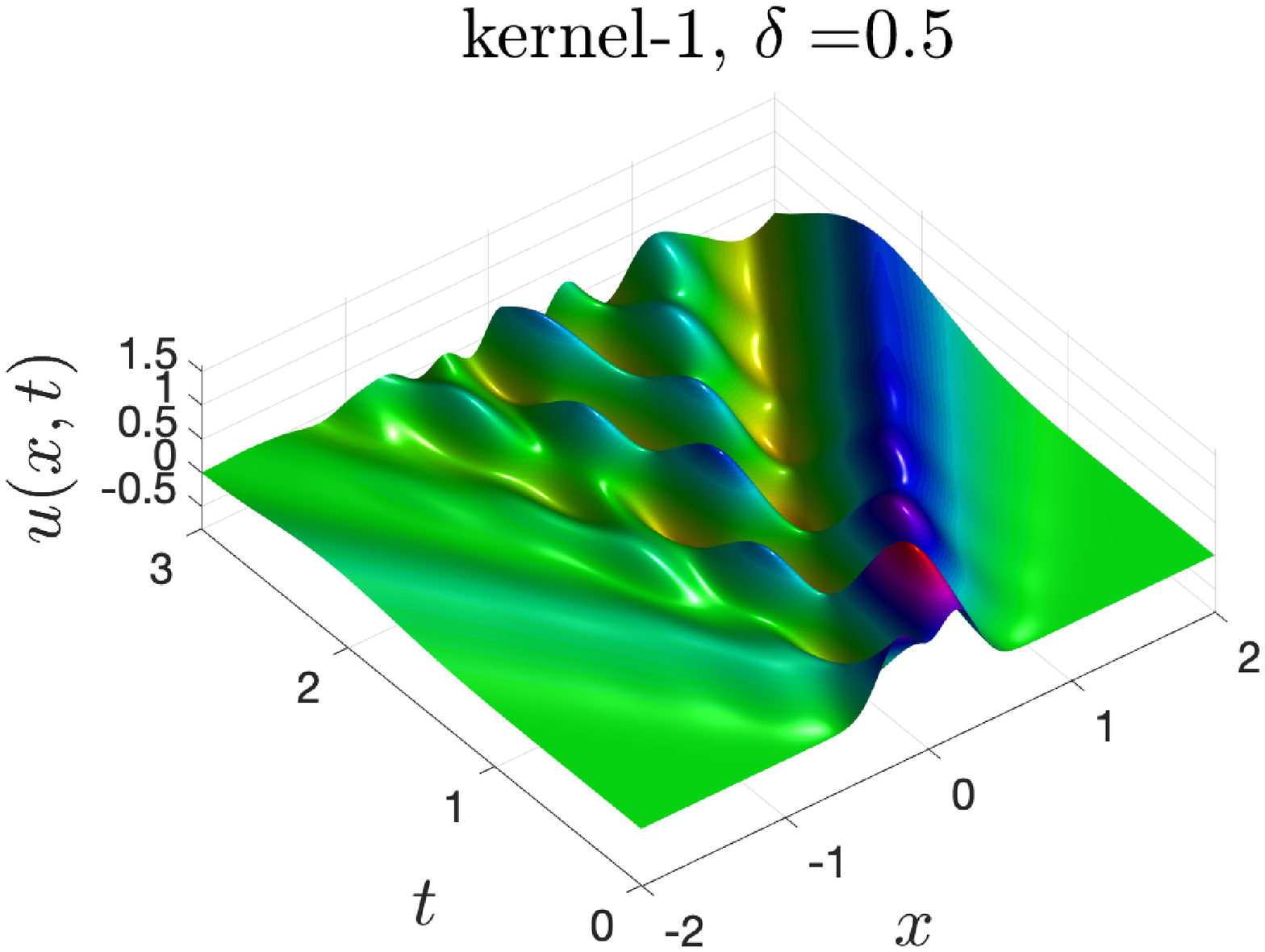}\vspace{0.1in} \\
	\includegraphics[width=0.46\textwidth]{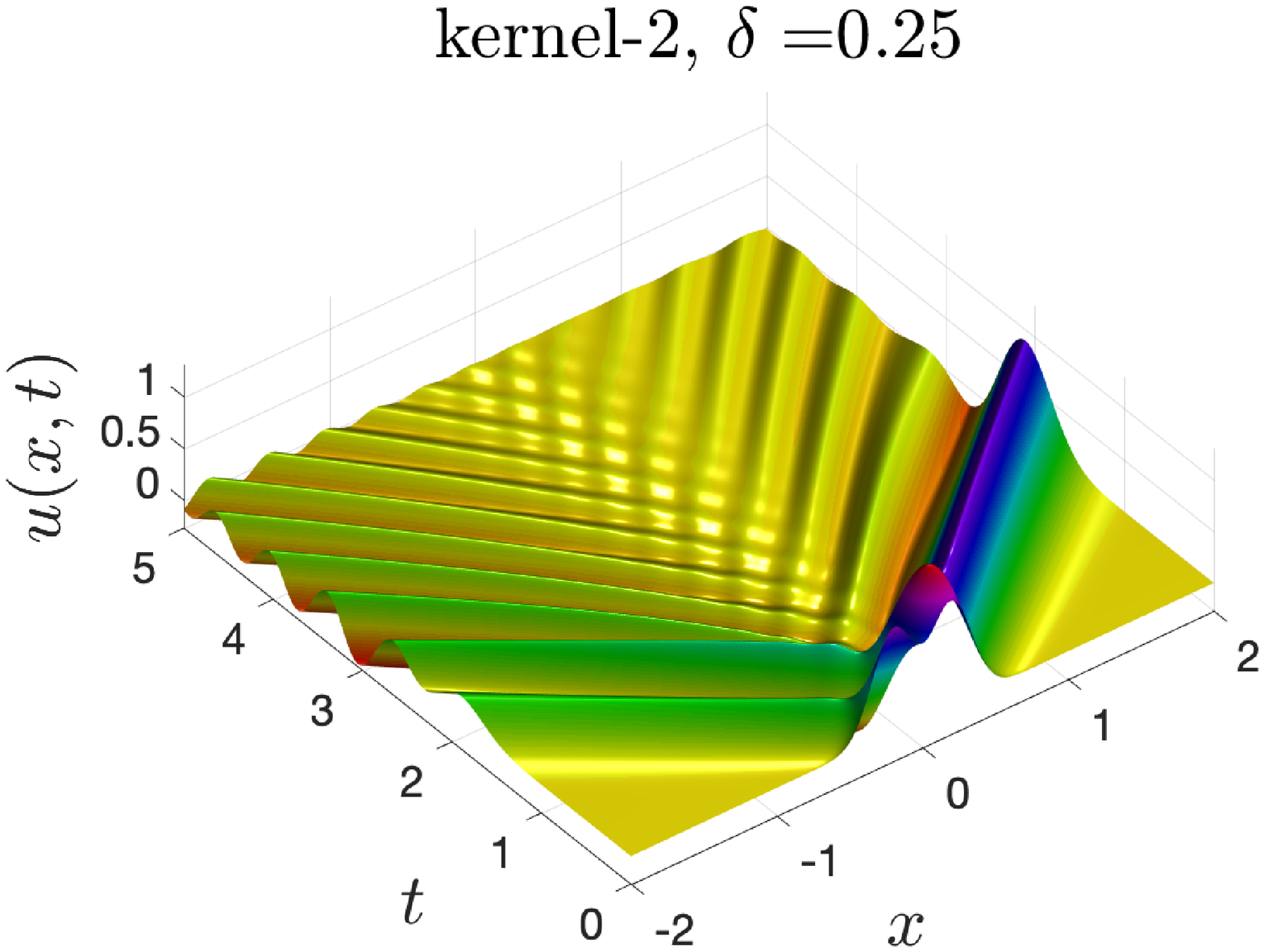}\hspace{0.2in}
	\includegraphics[width=0.46\textwidth]{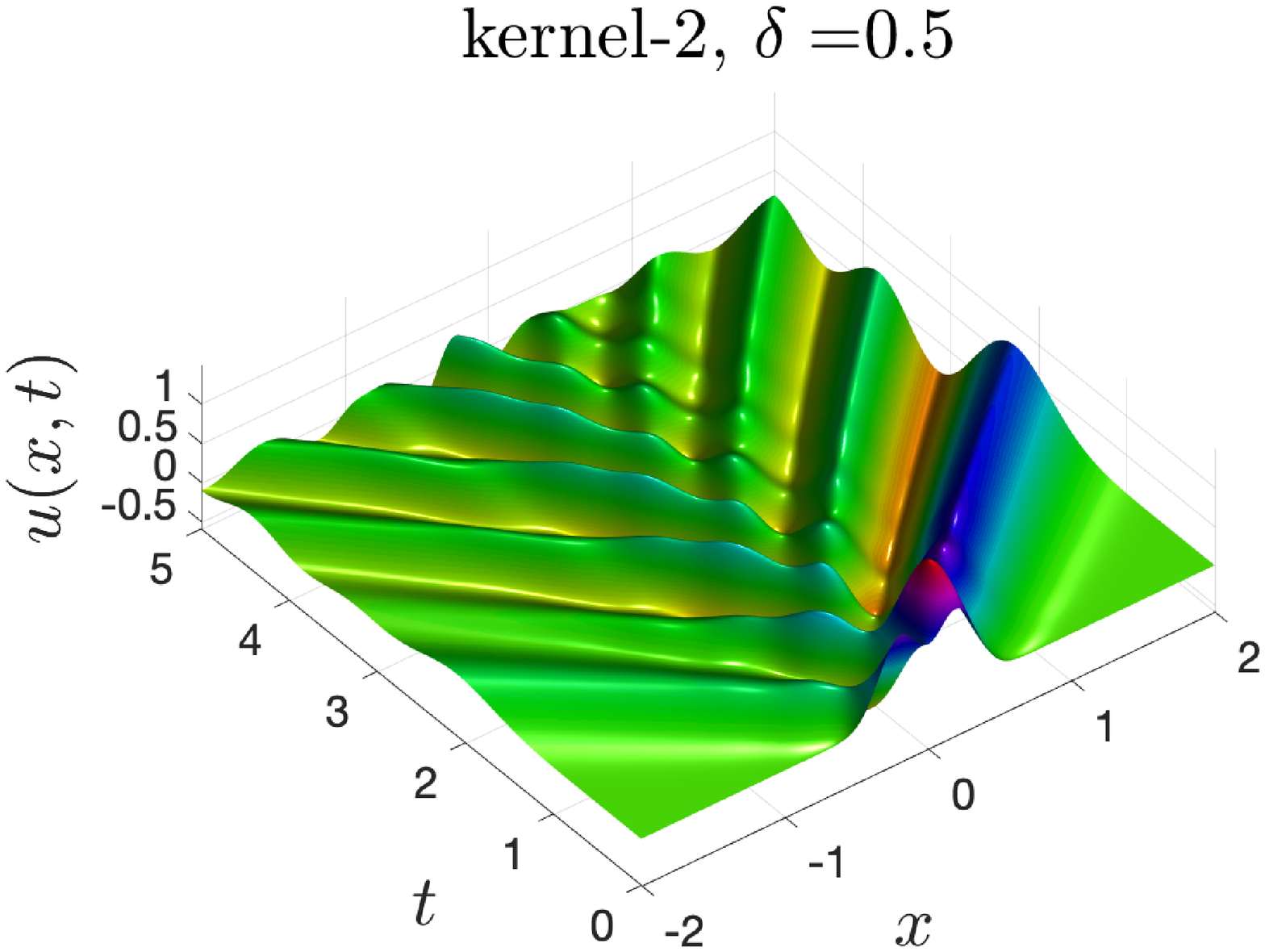}\vspace{0.1in}\\
	\includegraphics[width=0.46\textwidth]{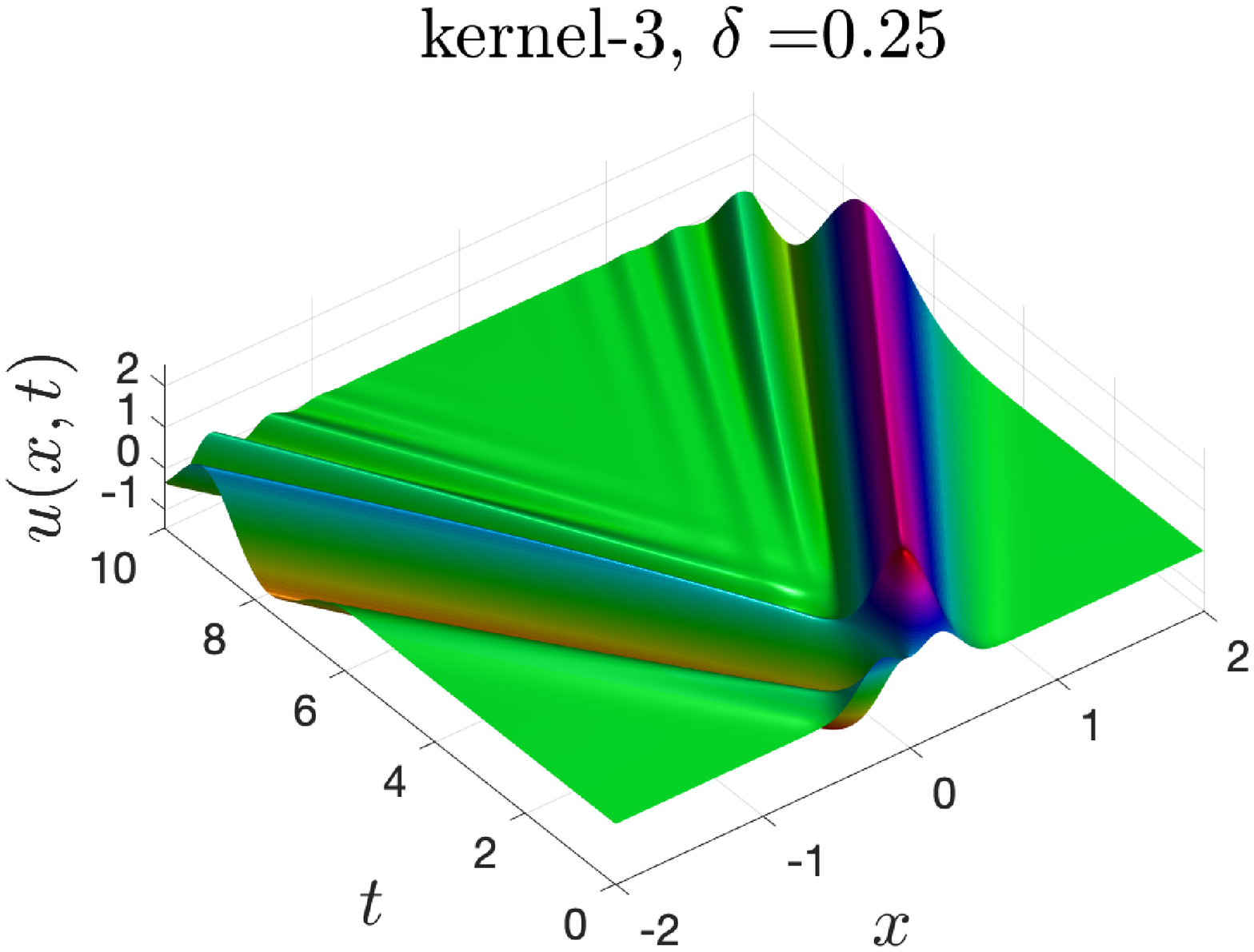}\hspace{0.2in}
	\includegraphics[width=0.46\textwidth]{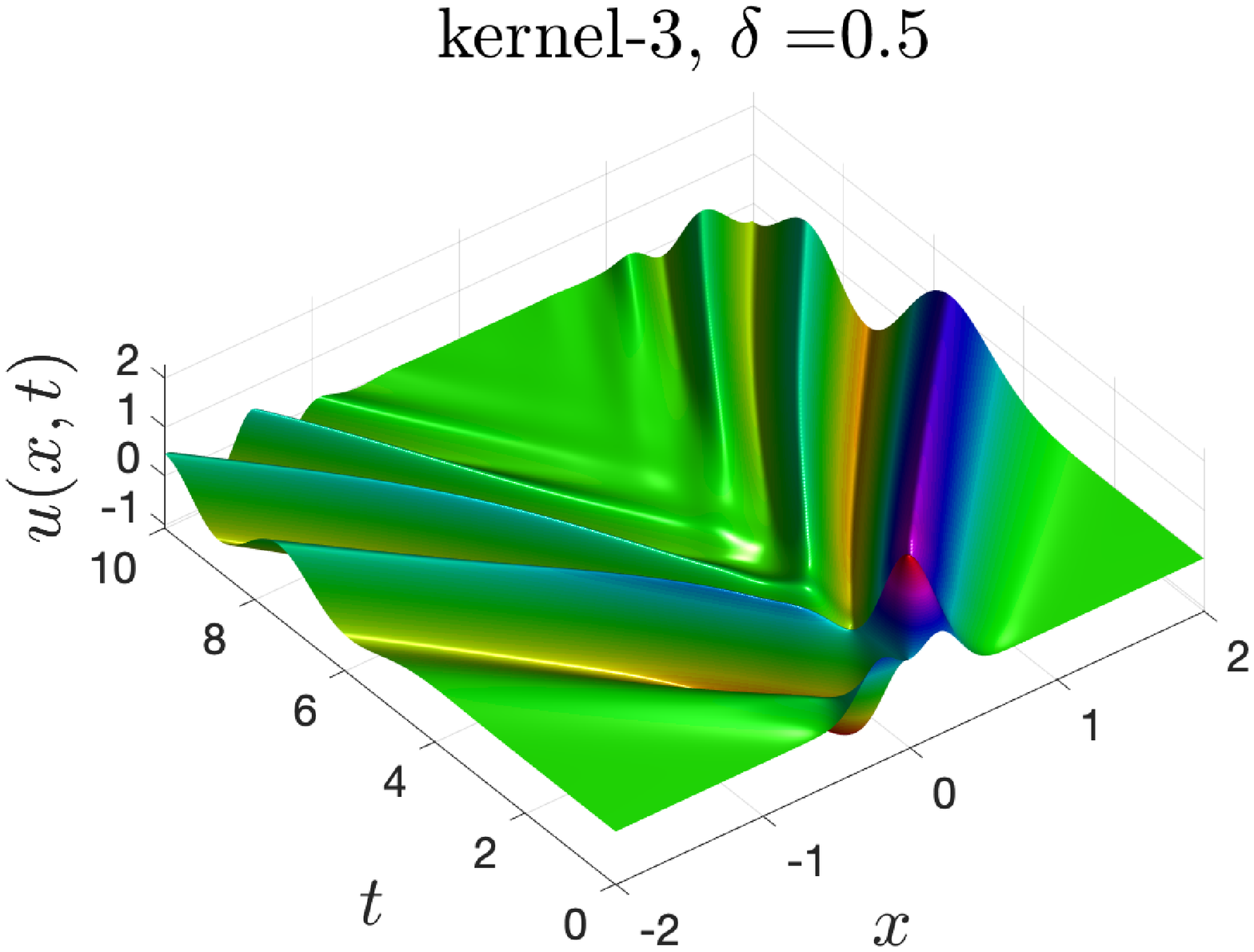}
\end{minipage}
\caption{(Example \ref{ex1}:) Evolution of numerical solutions.}
\label{ex1sol}
\end{figure}

To investigate the spatial convergence orders of various approximations such as linear, quadratic and cubic Lagrange interpolations, we set $\delta =1/8$, $T=2$ and fix $\tau=10^{-5}$, {$P=80000$}. The $L^2$-errors and convergence rates are shown in Figure \ref{error1d} by taking ${h}=[2^{-4},2^{-5},2^{-6},2^{-7}]$ for linear and quadratic cases, and $h=[1/24,1/48,1/72,1/96]$ for cubic case.
Here the ``exact" solutions are computed by pseudo-spectral method over a domain large enough as reference solutions. One can observe that linear interpolation scheme has the second-order convergence rate by comparing it with the second-order slope for all three kernels. And quadratic, cubic Lagrange interpolations have the forth-order convergence rate, expect in a special case where the quadratic interpolation scheme is used to solve the problem \eqref{NonModel} with the kernel-3. This is caused by the singularity of the kernel-3.  {We remark that the used time steps in all simulations satisfy the restriction given in \eqref{time_restriction}, but this restriction is not sharp, which can be relaxed in the future. }
%\rd{One can verify that the results are consistent with the accuracy of the classical interpolation theory for composite integration.}
% \begin{table}[H]
%  \centering
%\caption{ \label{ex1ratenon}
% (Example 1:) $L^2$-errors and convergence rates for 1D wave problems.}
%\begin{tabular}{|c|cc|cc|ccc|}
%\hline
%& \multicolumn{2}{c|}{Linear}  & \multicolumn{2}{c|}{Quadratic} & \multicolumn{3}{c|}{Cubic}  \\
%\hline
%$h$    & Error &  Rate  & Error &  Rate & $h$    &Error &  Rate  \\
%\hline
%$2^{-4}$   & 2.99e-02  &*& 1.83e-02 & * &$1/24$   & 2.23e-05 &* \\
%\hline
%$2^{-5}$ &7.34e-03 & 2.03 &1.13e-03 &4.02 & $1/48$   &1.52e-05 & 3.87   \\
%\hline
%$2^{-6}$  &1.83e-03  &2.01 & 7.04e-05& 4.00 & $1/72$   &3.09e-07 & 3.93  \\
%\hline
%$2^{-7}$ & 4.56e-04 & 2.00 & 4.41e-06& 4.00 &$1/96$   & 9.75e-08 & 4.01 \\
%\hline
%\end{tabular}
%\end{table}

% \begin{table}[H]
%  \centering
%\caption{ \label{ex1ratenon}
% (Example 1:) $L^2$-errors and convergence rates between numerical solutions and reference solutions for 1D wave problems.}
%\begin{tabular}{|c|cc|cc|ccc|}
%\hline
%& \multicolumn{2}{c|}{Linear}  & \multicolumn{2}{c|}{Quadratic} & \multicolumn{3}{c|}{Cubic}  \\
%\hline
%$h$    & Error &  Rate  & Error &  Rate & $h$    &Error &  Rate  \\
%\hline
%$2^{-4}$   & 4.11e-02  &*& 5.43e-03 & * &$1/24$   & 1.17e-03 &* \\
%\hline
%$2^{-5}$ &1.05e-02 & 1.97 &3.11e-04 &4.13 & $1/48$   &7.50e-05 & 3.97   \\
%\hline
%$2^{-6}$  &2.63e-03  &2.00 & 1.90e-05& 4.03 & $1/72$   &1.48e-05 & 4.00  \\
%\hline
%$2^{-7}$ & 6.56e-04 & 2.00 & 1.18e-06& 4.01 &$1/96$   & 4.70e-06 & 4.00 \\
%\hline
%\end{tabular}
%\end{table}

\begin{figure}[t!]
\centering
\includegraphics[width=0.7\textwidth]{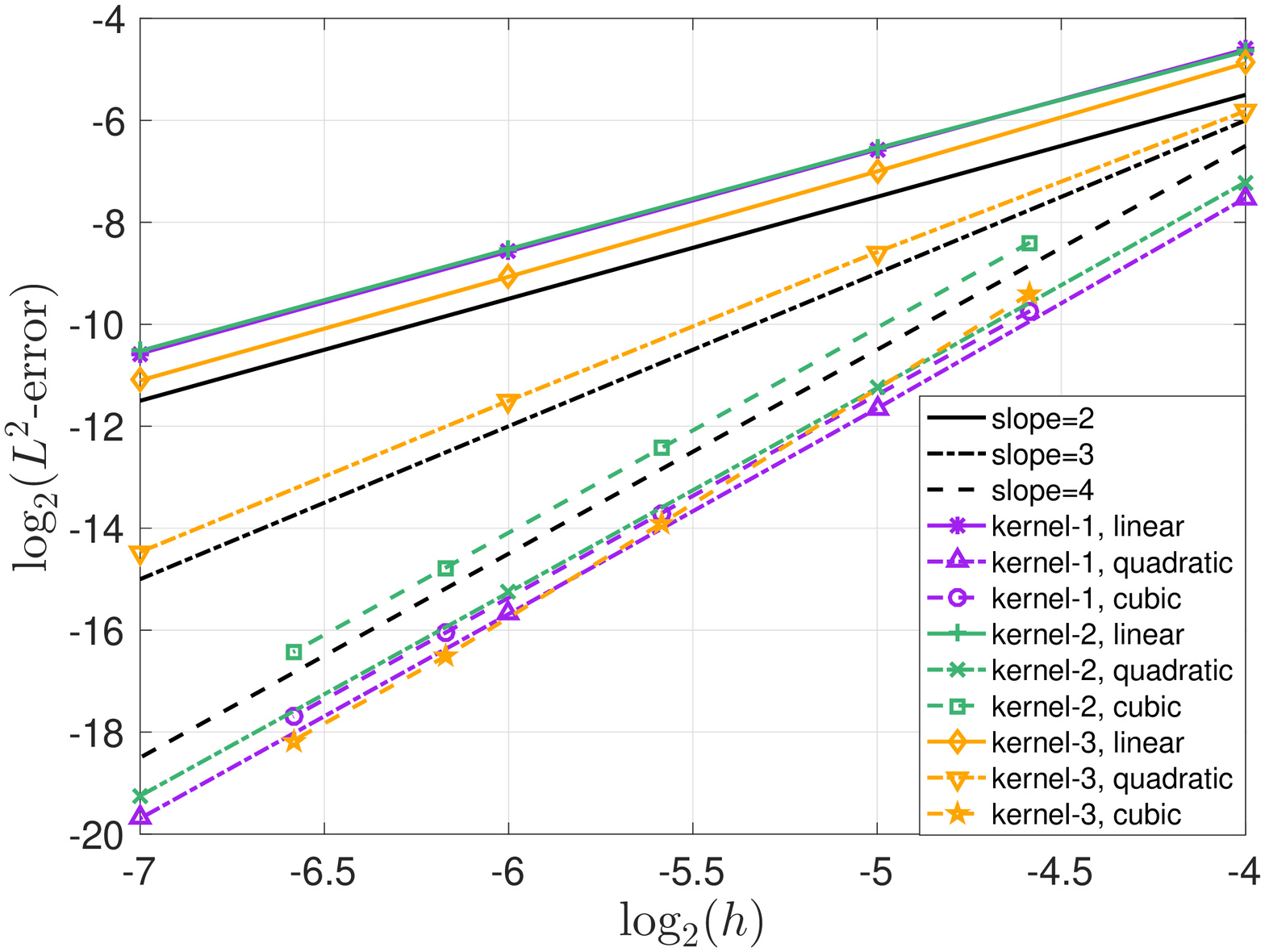}
\caption{(Example \ref{ex1}:) Convergence rates of different numerical schemes and kernels.}
\label{error1d}
\end{figure}

\end{example}

%\noindent{\textbf{Example 2.}}
\begin{example}\label{ex2}
In this example, we consider the two-dimensional problem \eqref{NonModel} with $f(\vx,t)=0$ and the initial values given as
\beqs
&&\varphi(\vx)=\exp(-25(\vx-0.2)^2)+\exp(-25(\vx+0.2)^2),\\
&&\psi(\vx)=\bm{0}.
\eeqs
 We choose the constant kernel function \eqref{kernel1} for $d=2$ and the Gaussian kernel
 $$\gamma(\bm{\alpha})=50\exp(-5\|\bm{\alpha}\|^2),\quad |\bm{\alpha}|_\infty\le\delta.$$
%$$\gamma(\bm{\alpha})=\frac{3}{2}\delta^{-4},\quad |\bm{\alpha}|\le\delta,$$
%which satisfies the second order moment condition \eqref{assump2}.
In the simulations, we take the computational domain $\Omega = (-1,1)^2$, $\delta=0.5$, $h=2^{-7}$, $\tau=10^{-3}$, and {$P=5000$}. Figure \ref{ex2solnon} shows the isolines of numerical solutions of scheme \eqref{FinalScheme} with the bilinear interpolation at times $T=0.1, 0.5, 1$, respectively. There is no obvious reflection caused by the boundary conditions for both two kernels. To show the error of the numerical solutions, we use the same strategy as that in Example \ref{ex1} to compute the reference solutions.
%Table \ref{ex2ratenon} shows the $L^2$-errors and the convergence orders, which is consistent with the theoretical analysis.
 {Figure \ref{error2d} shows the second-order and fourth-order convergence order in $L^2$-error by refining $h=[1/4,1/8,1/12,1/16,1/20]$, $\tau=[1/16,1/24,1/32,1/40,1/48]$ and $\tau=[1/16,1/36,1/64,1/100,1/144]$ for linear and quadratic interpolation cases, respectively, and taking the number of quadrature nodes as $P=[500,1000,2000,4000,5000]$. The convergence orders are consistent with the theoretical analysis.  }
 \begin{figure}[t!]
\centering
\begin{minipage}[]{\textwidth}
	\includegraphics[width=0.32\textwidth]{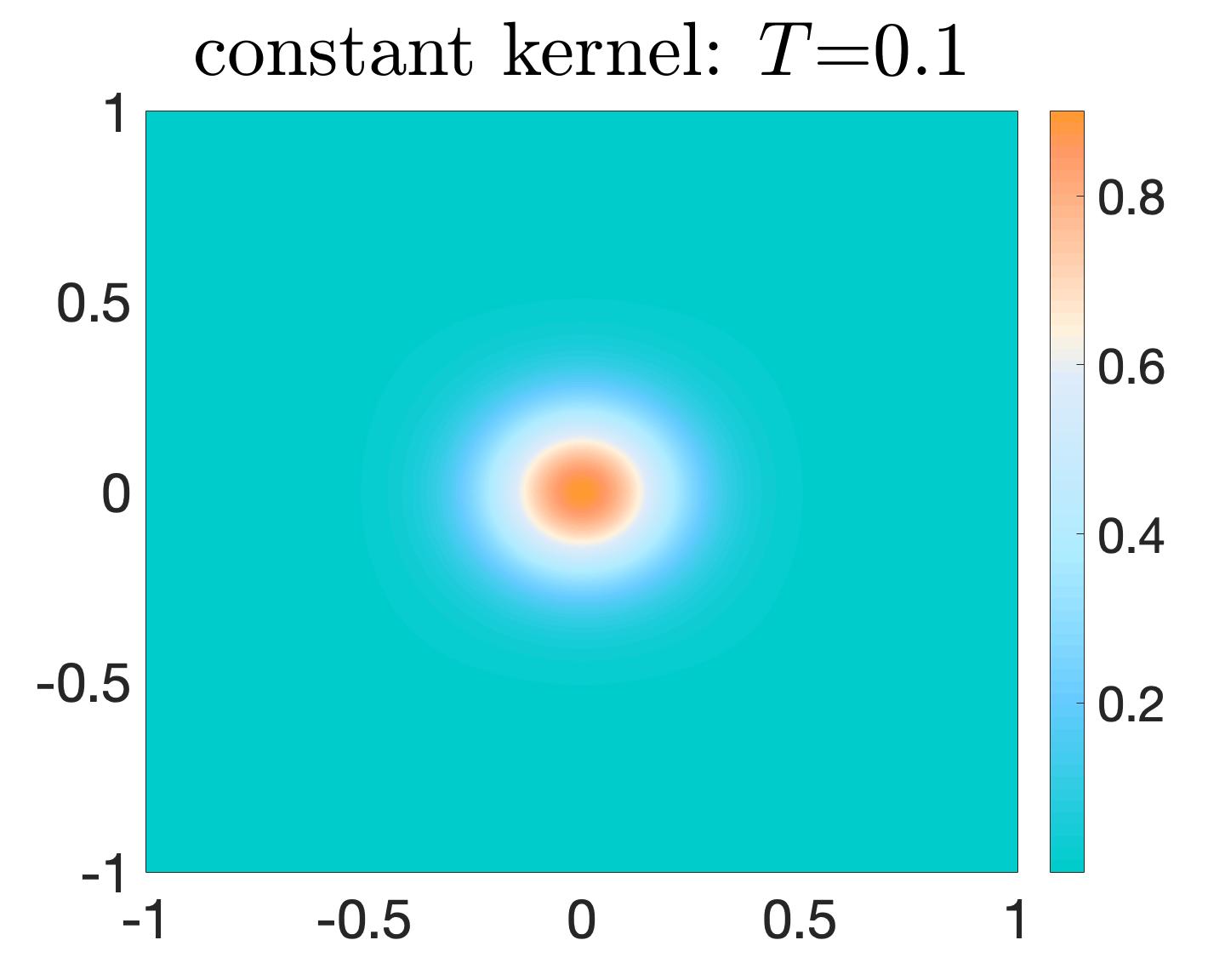}
	\includegraphics[width=0.32\textwidth]{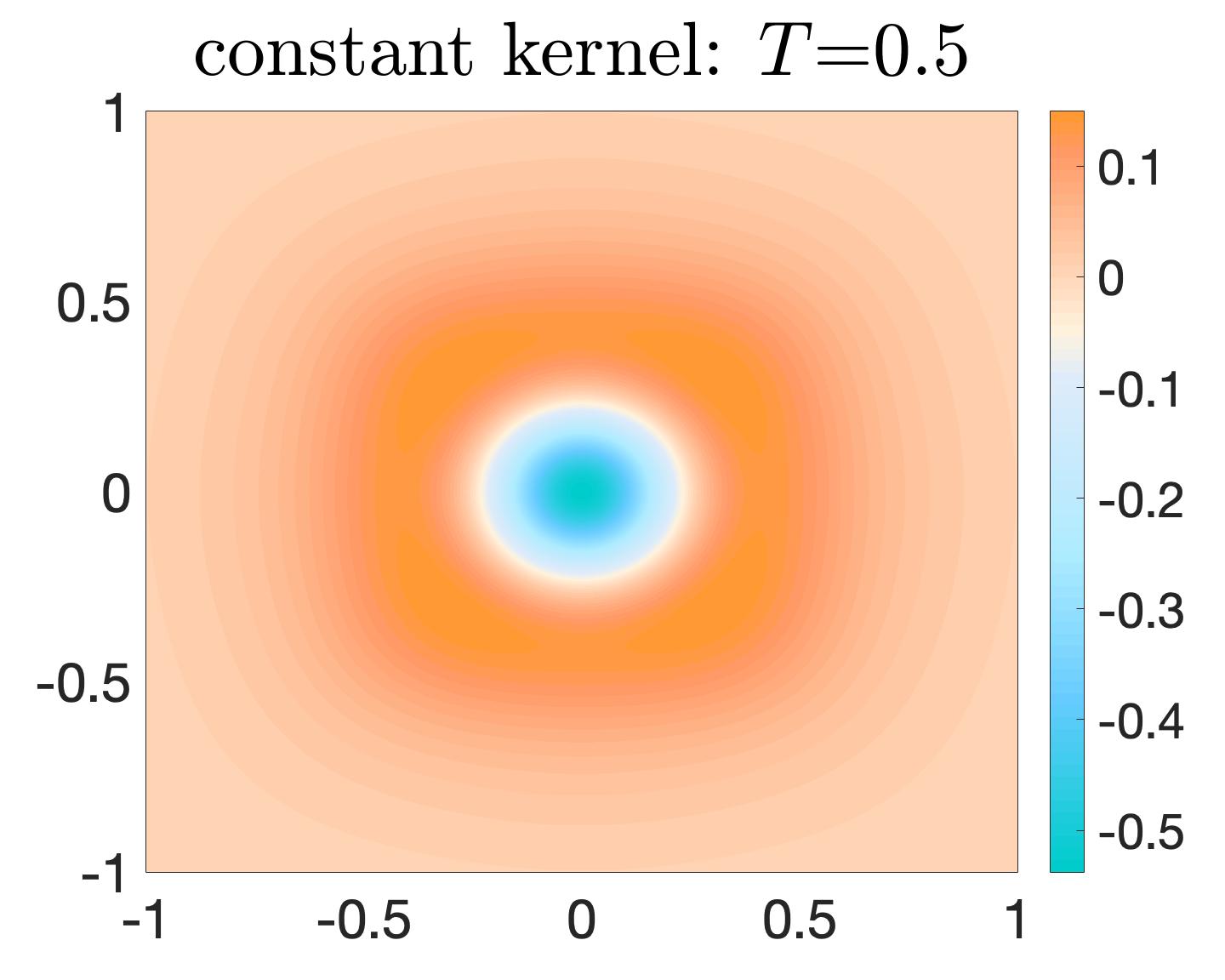}
	\includegraphics[width=0.32\textwidth]{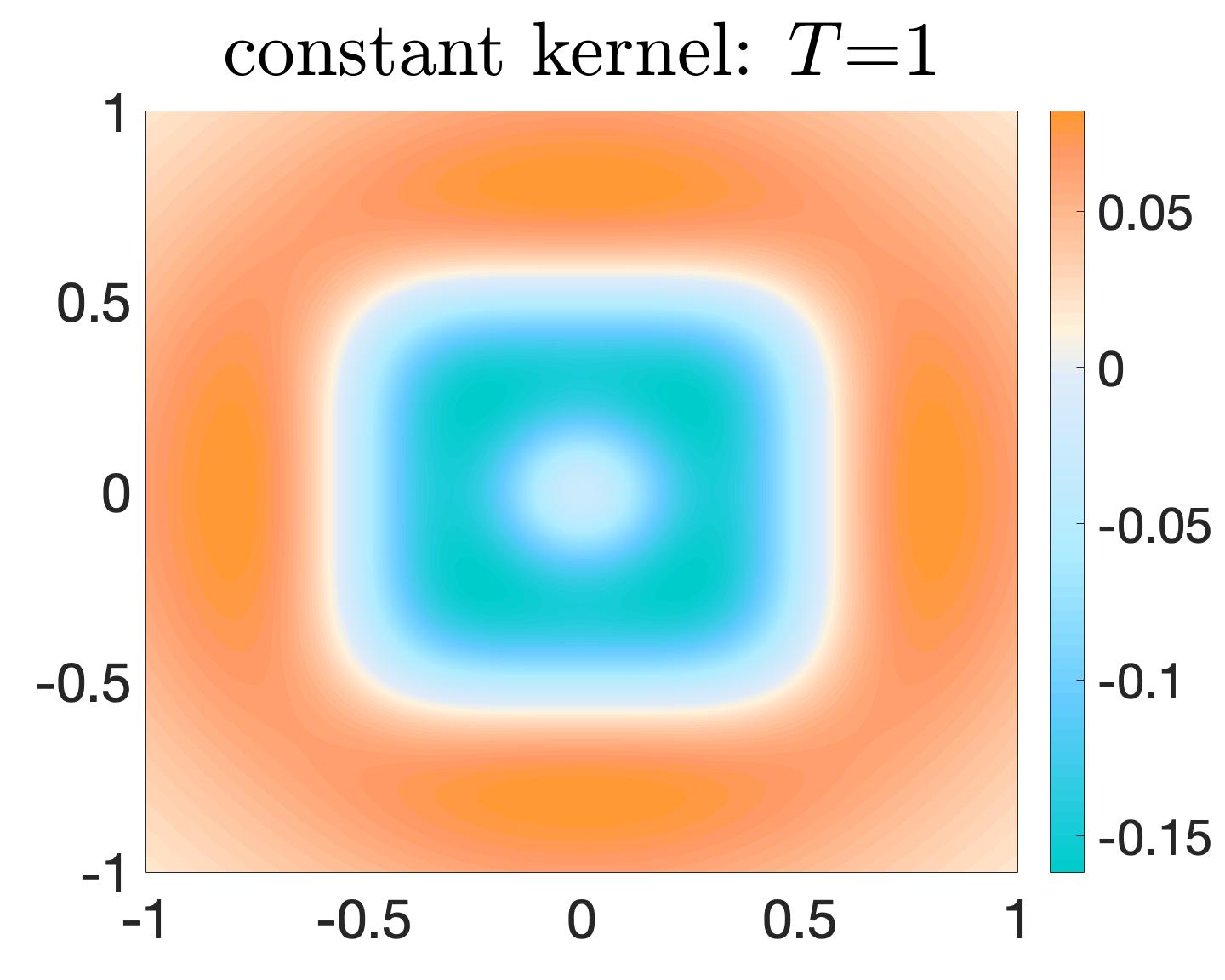}\\
	\includegraphics[width=0.32\textwidth]{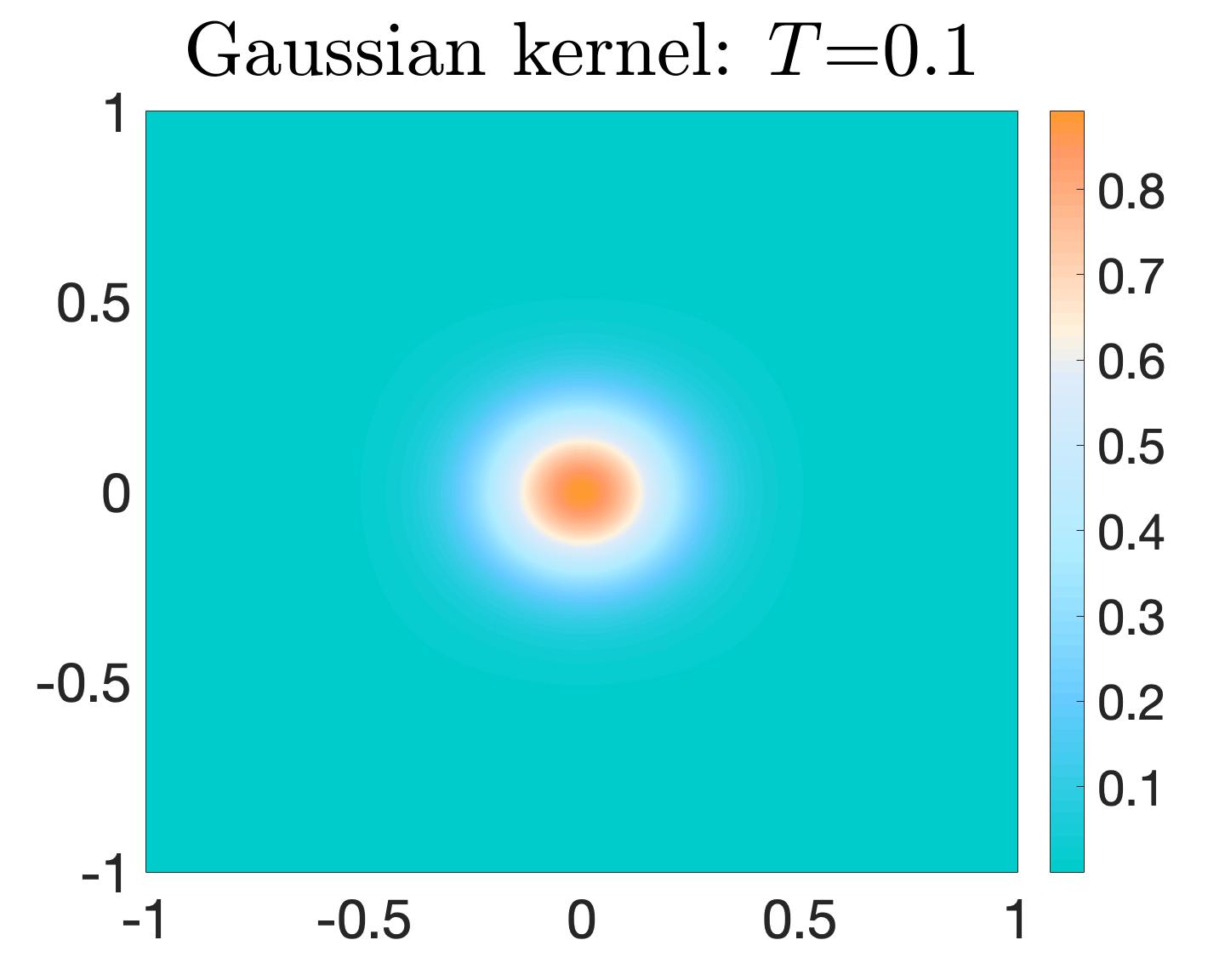}
	\includegraphics[width=0.32\textwidth]{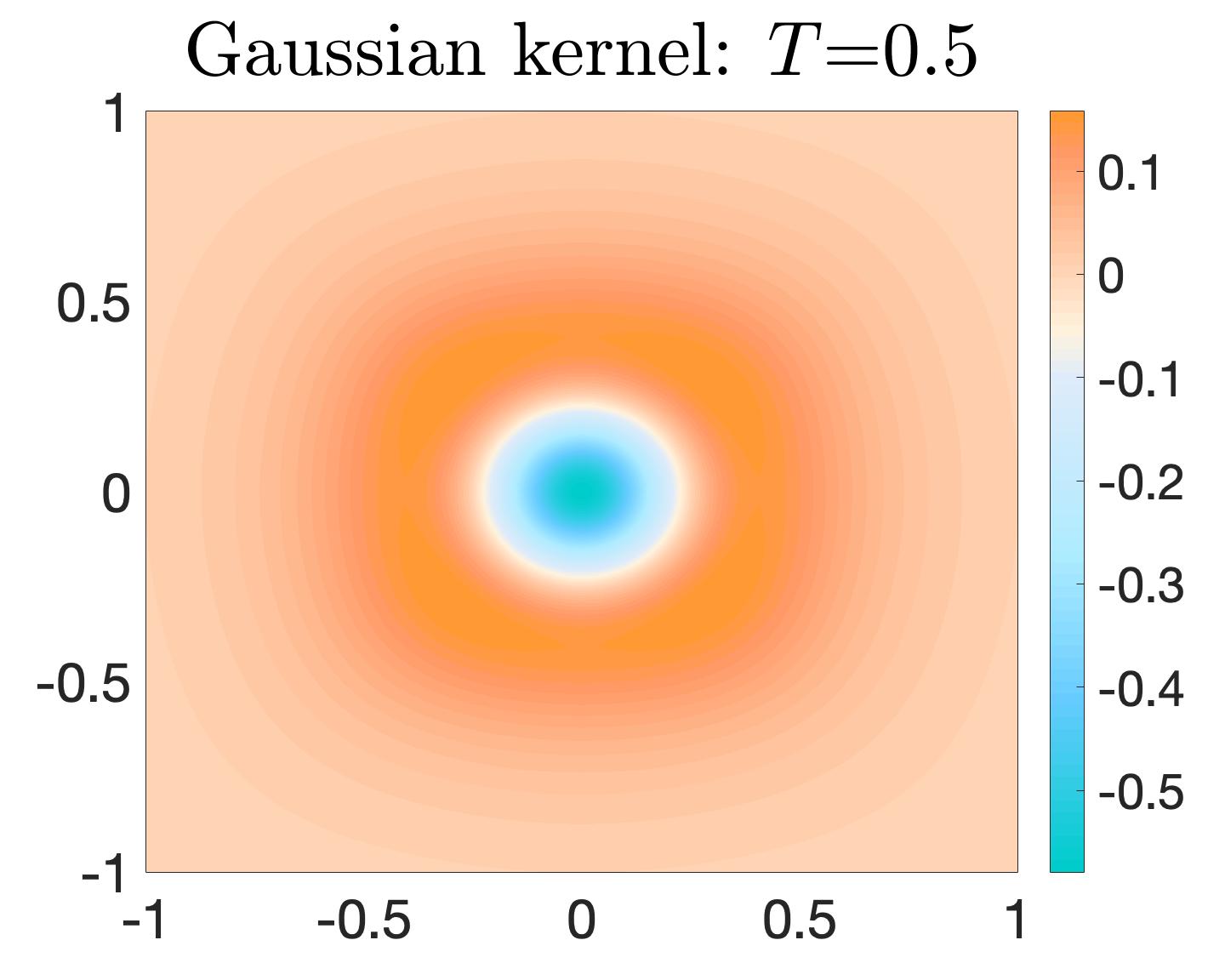}
	\includegraphics[width=0.32\textwidth]{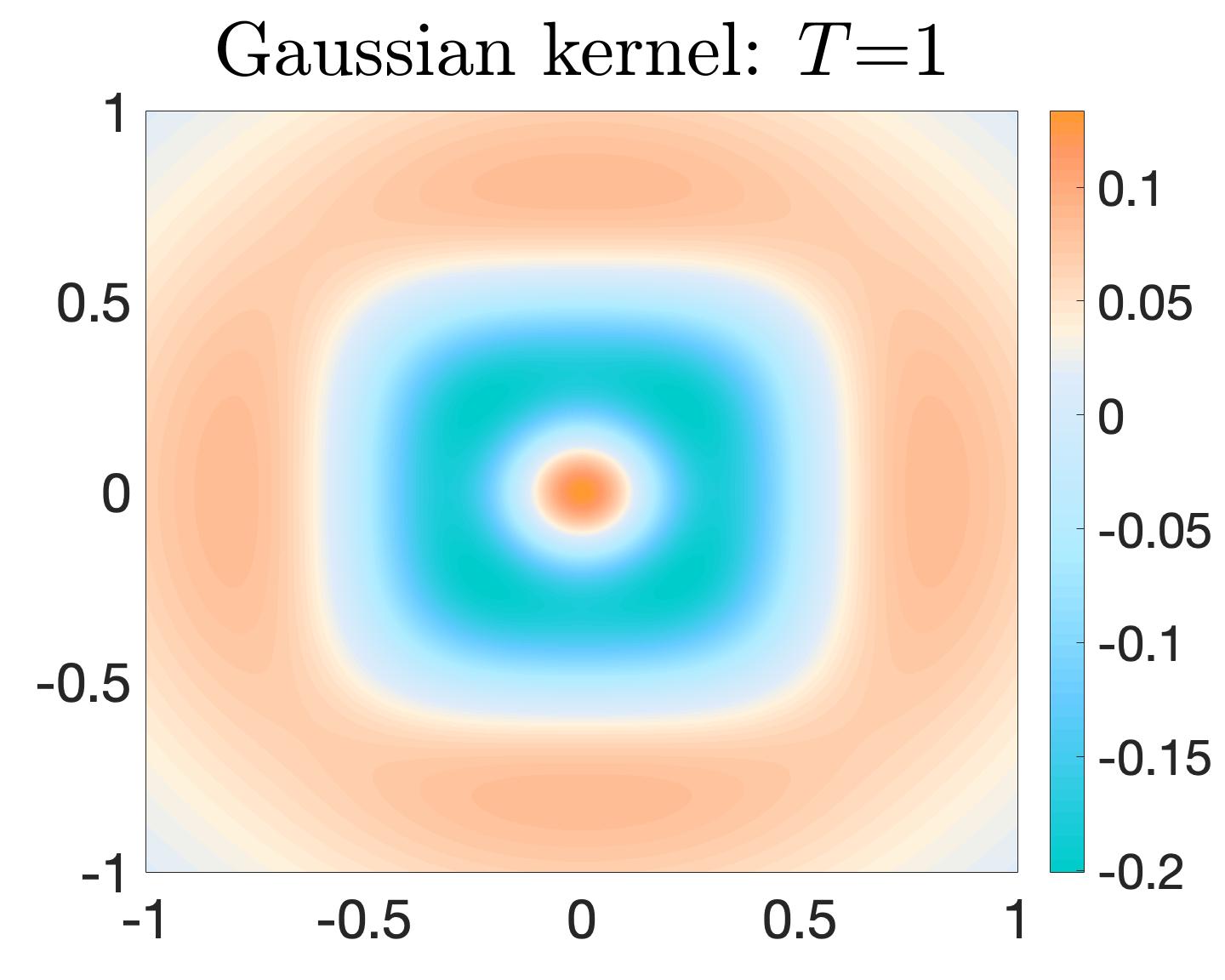}
\end{minipage}
\caption{(Example \ref{ex2}:) Isolines of numerical solutions at $T=0.1,0.5,1$.}
\label{ex2solnon}
\end{figure}

\begin{figure}[t!]
\centering
\includegraphics[width=0.7\textwidth]{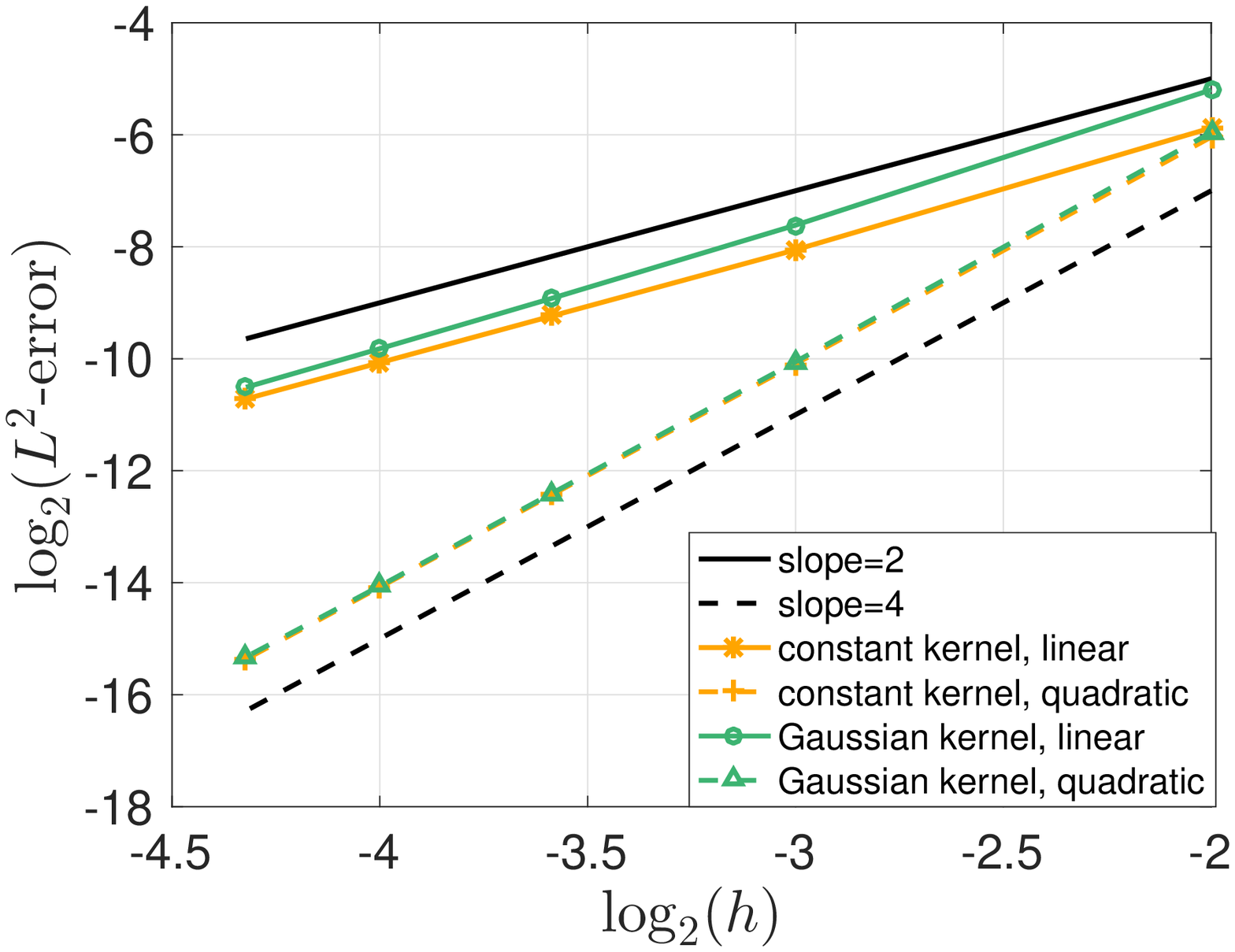}
\caption{(Example \ref{ex2}:) Convergence rates of different numerical schemes and kernels.}
\label{error2d}
\end{figure}

%Again, we investigate the AC property of numerical scheme \eqref{FinalScheme} using the same method as Example \ref{ex1} by fixing $\delta=2h$. %And only the problem with the constant kernel, which satisfies the second order moment condition \eqref{assump2}, is considered.
%%this is the convergence behavior as both $\delta$ and $h \rightarrow 0$. In the simulation, We fix $\delta=2h$, $T=0.1$ and $\tau=10^{-3}$.
%We compute the numerical solutions of schemes \eqref{FinalScheme} with the linear and quadratic interpolations by taking $h=[1/8,1/10,1/12,1/14,1/16]$, $\tau=10^{-3}$,  {$P=10000$} and $T=0.1$ for constant kernels. Figure \ref{ex2_AC} shows the $L^2$-errors between numerical solutions  and the corresponding local solution. One can see that both two approximations have the second-order convergence rates with respect to $\delta$, which is consistent to the theoretical analysis in \cite{tian2013analysis,tian2020asymptotically}, again.
%\begin{figure}[t!]
%\centering
%\includegraphics[width=0.7\textwidth]{ex2_AC.eps}
%\caption{(Example \ref{ex2}:) Convergence of $L^2$-errors between solutions of numerical scheme \eqref{FinalScheme} for 2D constant kernel and corresponding local solution.}
%\label{ex2_AC}
%\end{figure}
\end{example}

%\begin{example}\label{ex3}
%The approach proposed in this paper is also
% \begin{table}[H]
% \centering
%\caption{ \label{ex2ratelocal}
% (Example \ref{ex3}:) $L^2$-errors and convergence rates with different numerical schemes for 2D local problem.}
%\begin{tabular}{|c|cc|cc|ccc|}
%\hline
%   & \multicolumn{2}{c|}{$L=1$}  & \multicolumn{2}{c|}{$L=2$} & \multicolumn{3}{c|}{$L=3$} \\
%\hline
%$h$    & Error &  Rate  & Error &  Rate & $h$ & Error &  Rate \\
%\hline
%$1/8$ &2.02e-02 & *          & 1.95e-02 &  *        &$1/8$ & 1.54e-02&  * \\
%
%$ 1/12$   & 8.82e-03 & 2.06 & 3.64e-03 & 4.14   &$ 1/10$ &  3.71e-03& 6.37 \\
%
%$ 1/16$  & 4.83e-03  & 2.09 & 1.12e-03 & 4.08  &$ 1/12$ & 1.21e-03 & 6.15 \\
%
%$ 1/20$   & 3.00e-03 & 2.09 & 4.55e-04  & 4.05    &$ 1/14$ & 4.72e-04 & 6.09\\
%
%$ 1/24$   & 2.08e-03 & 2.08 & 2.18e-04 & 4.03   & $ 1/16$ & 2.12e-04 & 6.02 \\
%\hline
%\end{tabular}
%\end{table}
%Example \ref{ex3}
%\end{example}

\section{Conclusion}
In this paper we considered the sharp error estimate of arbitrarily high-order schemes in space for multi-dimensional nonlocal wave equations on unbounded domains. To this end, we first approximated the nonlocal operator with arbitrarily high-order quadrature-based difference schemes, and discretized the time direction with the explicit difference scheme to have a fully discrete infinity system. After that, we used the methodology in \cite{du2018nonlocal,du2018numerical}  to achieve the DtD-type ABCs for the resulting infinity system, and further presented the formula of nonlocal Neumann data based on the discrete nonlocal Green's first identity, and finally obtained the DtN-type ABCs. The DtN-type ABCs are available to reduce the infinite system to a finite discrete system, whose solution is equivalent to that of the infinite system confined on the bounded computational domain. On the other hand, the DtN-type ABCs are also available to present the stability analysis for the reduced finite discrete system. In the practical simulation, the convolution kernel in time arose from the inverse $z$-transform can be approximated with high-order accuracy, i.e., the resulting error can be small enough such that it does not bring the loss of the optimal convergence order.   Finally, the efficiency and accuracy of our proposed approach were verified by numerical examples. And we point out that the proposed method above can be extended to solve the classical local wave problems on unbounded domains with arbitrarily high-order schemes in spatial direction.

It is well-known that the direct evaluation of the convolution kernel in \eqref{DtD0} is quite expensive.
%, it still keeps open how to develop its fast algorithm.
For the local problems, there are many works on the fast evaluation of ABCs \citep[see, e.g.,][]{zheng2007approximation, jiang2004fast, arnold2003discrete, li2007numerical, sun2020fast}. While the operator $\K$ in nonlocal models is more complicated than it in local models, it is difficult to achieve a fast algorithm to the inverse $z$-transform.
%However, there are few related works for the nonlocal problems.
Recently, \citet{zheng2020stability} have developed a fast algorithm by utilizing the discretized contour integrals developed in \cite{lopez2005fast} for solving the nonlocal heat equation on unbounded domains, but the technique is nontrivial for the wave problem.
Thus, further efforts are required to address the fast evaluation of ABCs for nonlocal wave problems.

 {
Additionally, in this work, we have achieved high-order accuracy in space,  but only have the second-order accuracy in time. It is natural to ask whether the high-order accuracy in time can be achieved. Fortunately, for the high-order scheme obtained by the modified equation technique \citep[see, e.g.,][]{shubin1987modified}, which is usually adopted to deal with the wave equations, the method of deriving ABCs in this paper seems to be applicable. However, how to analyze the stability of the scheme requires more detailed discussions. In future work, we will extend our method to high-order schemes in time.
}

\section*{Acknowledgements}

Jerry Zhijian Yang is supported by National Science Foundation of China (No. 12071362 and 11671312), the National Key Research and Development Program of China (No. 2020YFA0714200), the Natural Science Foundation of Hubei Province (No. 2019CFA007).
Jiwei Zhang is partially supported by NSFC under grant Nos. 11771035 and 12171376, 2020-JCJQ- ZD-029 and NSAF U1930402.
The numerical simulations in this work have been done on the supercomputing system in the Supercomputing Center of Wuhan University.

 {
\section*{Appendix}
{\bf The proof of Lemma \ref{trun_err}. }
First we consider the case of one-dimension. We review the domain division given in section 2
%Let $\delta/h=Cp$, where $C$ is a positive integer. Dividing the integral domain $B_\delta({x_k})$ into $2L/p$ equal parts centered at $x_{k}$. Denote
$$T^k_i=[x_{k}+((i-1)p-L)h, x_{k}+(ip-L)h], \quad i=1,2,\dots,2L/p,$$
then $\displaystyle B_\delta({x_k}) = \cup_i T^k_i$. The interpolation points in every subdomain $T^k_i$ are given as
$$
s_{i,j} = x_{k}+((i-1)p-L+j)h, \quad j=0,1,\cdots, p.
$$

For integral
$$I(f)=\int_{B_{\delta}(0)}f(s)w(s)\gamma(s)ds,$$
we consider the numerical integration for $I(f)$
\beq\label{Ih}
I_{h,p}(f)=\sum_{i}\int_{T^0_i} \mathcal{I}_{i,p}[f](s)w(s)\gamma(s)ds,
\eeq
where $\mathcal{I}_{i,p}$ represents the $p$th-degree Lagrange interpolation operator on $T^0_i$. For simplicity, we denote $T_i:=T^0_i$. According to the interpolation error of the Lagrange interpolation formula, one has
\begin{align}
\mathcal{R}[f] =I(f)- I_{h,p}(f)
=&\sum\limits_{i}\int_{T_i} (f(s)-\mathcal{I}_{i,p}[f](s))w(s)\gamma(s)ds \notag \\
=&\sum\limits_{i}\int_{T_i} \frac{f^{(p+1)}(\xi_i)}{(p+1)!}\prod\limits_{j=0}^{p}(s-s_{i,j})w(s)\gamma(s)ds,
\end{align}
where $\xi_i\in T_i$. Obviously, $\mathcal{R}[f]=0$ for polynomials with degree less than or equal to $p$. Moreover, when $p$ is even, numerical integration \eqref{Ih} is also accurate for polynomials with degree of $p+1$.
Considering $f(s)=s^{p+1}$, one has
%显然其对不高于$p$次的多项式是精确成立的，即数值积分\eqref{Ih}具有$p$次代数精度。特别地，当$p$为偶数时，对$p+1$次多项式$f(s)=s^{p+1}$，有
\begin{align*}
\mathcal{R}[f]
=\sum\limits_{i}\int_{T_i}\prod\limits_{j=0}^{p}(s-s_{i,j})w(s)\gamma(s)ds.
\end{align*}
The above error is zero since the integral domain is symmetric about the origin and the integrand is an odd function.
%借助积分区域$B_\delta(0)$关于原点的对称性，可知$\mathcal{R}[f]=0$。这说明数值积分\eqref{Ih}有$p+1$次代数精度。

Based on the symmetry of the kernel, the nonlocal operator \eqref{nonop} can be rewritten as
\beqs
\L_\delta u(x_k) = \frac12\int_{B_{\delta}(0)} \frac{2u(x_k)-u(x_k+s)-u(x_k-s)}{w(s)}w(s)\gamma(s)ds.
\eeqs
Denote
\beqs
G:=G(s; x_k) = \frac{2u(x_k)-u(x_k+s)-u(x_k-s)}{w(s)},
\eeqs
then the numerical scheme \eqref{a} is
\beq \label{Lh}
\L_{\delta,h} u(x_k) = \sum_{i}\int_{T_i} \mathcal{I}_{i,p}[G](s; x_k)w(s)\gamma(s)ds.
\eeq
When $p$ is odd, we construct the auxiliary polynomial with degree of $p$
%构造$p+1$次代数多项式
$$H(s; x_k) = -2\sum_{m=1}^{(p+1)/2}\frac{s^{2m}u^{(2m)}(x_k)}{(2m)!w(s)}.$$
Let
\beqs
J(s; x_k) =G(s; x_k)-H(s; x_k).
\eeqs
According to the Taylor's expansion, one yields
\beqs
J(s; x_k) = %\frac{2s^{p+3}u^{(p+3)}(x_k)}{(p+3)!w(s)}+\mathcal{O}(s^{p+3}).
\frac{-2}{w(s)}\int_0^s \frac{(s-t)^{p+2}}{(p+2)!}u^{(p+3)}(x_k+t)dt.
\eeqs
Further, we calculate the $p$th-order derivate of $J(s;x_k)$ to have
\beq\label{J_p}
|J^{(p+1)}(s;x_k)| \le C(p)|u|_\infty |s|.
\eeq

The truncation error of the approximation \eqref{a} is given as
\begin{align}
\left| \L_{\delta}u(x_k)- \L_{\delta,h}u(x_k) \right|
=& \left| \frac12\sum_i\int_{T_i}\left( G-\mathcal{I}_{i,p}[G] \right) w(s)\gamma(s)ds \right| \notag \\
=& \left| \frac12\sum_i\int_{T_i} \left((\mathcal{I}_{i,p}[H]-H)-(\mathcal{I}_{i,p}[J]-J) \right) w(s)\gamma(s)ds \right| \notag\\
\le& \frac12\sum_i\int_{T_i} \left|\mathcal{I}_{i,p}[H]-H\right| w(s)\gamma(s)ds+ \frac12\sum_i\int_{T_i} \left|\mathcal{I}_{i,p}[J]-J\right| w(s)\gamma(s)ds \notag\\
:=& E_1 +E_2,
\end{align}
where $E_1=0$ since $H$ is a polynomial of degree $p$. Next we estimate $E_2$
%由于$H$为$p$次多项式，故$E_1=0$。下面估计$E_2$：
\begin{align}
E_2 &= \frac12\sum_i\int_{T_i}\left| \frac{J^{(p+1)}(\xi_i;x_k)}{(p+1)!}\prod\limits_{j=0}^{p}(s-s_{i,j}) \right| w(s)\gamma(s)ds \notag\\
%& \le C(p)\delta \sum_i\int_{T_i}\left| u^{(p+3)}(x_k)\prod\limits_{j=0}^{p}(s-s_{i,j}) \right| w(s)\gamma(s)ds \notag\\
%& \le C(p)\delta h^{p+1}|u^{(p+3)}|_{\infty} \int_{B_{\delta}(0)}w(s)\gamma(s)ds.
& \le C(p)|u^{(p+3)}|_\infty \sum_i\int_{T_i}\left| \xi_i \prod\limits_{j=0}^{p}(s-s_{i,j}) \right| w(s)\gamma(s)ds \notag\\
& \le C(p)\delta|u^{(p+3)}|_{\infty} h^{p+1} \int_{B_{\delta}(0)}w(s)\gamma(s)ds.
\end{align}
Then if $u\in C_b^{p+3}(\R^d)$ and $w\gamma$ is integral on domain $B_{\delta}(0)$, the approximation error of \eqref{Lh} is $\O(h^{p+1})$ and the estimate constant $C$ is independent of $h$.

When $p$ is even, the numerical error of $\O(h^{p+2})$ can be achieved based on the fact that the numerical integration \eqref{Ih} has the $(p+1)$th-degree of exactness.
%并且$w\gamma$在$B_{\delta}(0)$上可积时，数值格式\eqref{Lh}的截断误差阶为$p+1$阶。当$p$为偶数时，利用数值积分\eqref{Ih}具有$p+1$次代数精度这一性质，还可以得到更高一阶的误差估计。
%首先，易见当$p$为偶数时，利用积分区域的对称性，格式\eqref{Lh}有$p+1$次代数精度，即对
%%其证明思路可参考文献\cite{stoer2013introduction}。
%$G(s)=s^{p+1}$，有 % 这样写很不合理
%\begin{align}
%\L_{\delta}u(x_k)- \L_{\delta,h}u(x_k)
%&= \sum\limits_{i}\int_{T_i} \prod\limits_{j=0}^{p}(s-s_{i,j})w(s)\gamma(s)ds \notag\\
%&= 0.
%\end{align}
We construct the $(p+1)$th-degree polynomial $H_i(s;x_k)$ on $T_i$, which satisfies
%在$T_i$上构造$p+1$次多项式$H_i(s;x_k)$使其满足
\begin{align*}
&H_i(s_j; x_k)=G(s_{i,j},x_k),\quad j=0,1,\cdots, p; \\
&H_i'(s_{i,*};x_k)=G(s_{i,*};x_k),\quad s_{i,*}=x_{k}+((i-1/2)p-L)h ~(\text{midpoint of $T_i$}).
\end{align*}
According to the error of the Hermite interpolation formula, one has
%由Hermite插值公式的误差余项知，
\beqs
J_i(s;x_k):=G(s;x_k)-H_i(s; x_k)= \frac{G^{(p+2)}(\xi_i;x_k)}{(p+2)!}(s-s_{i,*})\prod_{j=0}^{p}(s-s_{i,j}), \quad \xi_i\in T_i.
\eeqs
Noting that the value of $H_i(s;x_k)$ only depends on the values of $G$ on the interpolation points, one has
%由于$H_i(s;x_k)$只取决于$G$在插值节点上的值，故有
\begin{align*}
\sum_i\int_{T_i}\mathcal{I}_{i,p}[G](s;x_k)w(s)\gamma(s)ds &=\sum_i\int_{T_i}\mathcal{I}_{i,p}[H_i](s;x_k)w(s)\gamma(s)ds\\
&=\sum_i\int_{T_i}H_i(s;x_k)w(s)\gamma(s)ds.
\end{align*}
%于是有
%\begin{align}
%&\left| \L_{\delta}u(x_k)- \L_{\delta,h}u(x_k) \right| \notag \\
%=&\frac12 \left| \int_{B_{\delta}(0)}\left( G(s;x_k)-I_{h,p}(G(s;,x_k)) \right) w(s)\gamma(s)ds \right| \notag \\
%=&\frac12 \left| \int_{B_{\delta}(0)}\left((I_{h,p}(H)-H)-(I_{h,p}(J)-J) \right) w(s)\gamma(s)ds \right| \notag\\
%\le& \frac12 \int_{B_{\delta}(0)}\left|I_{h,p}(H)-H\right| w(s)\gamma(s)ds+\frac12 \int_{B_{\delta}(0)}\left|I_{h,p}(J)-J\right| w(s)\gamma(s)ds \notag\\
%:=& E_1 +E_2.
%\end{align}
Finally one yields
\begin{align}
\left| \L_{\delta}u(x_k)- \L_{\delta,h}u(x_k) \right|
=& \left| \frac12\sum_i\int_{T_i} \left( G(s;x_k)-\mathcal{I}_{i,p}[G](s;x_k) \right) w(s)\gamma(s)ds \right| \notag \\
=&\left| \frac12\sum_i\int_{T_i}\left( G(s;x_k)-H_i(s;x_k) \right) w(s)\gamma(s)ds \right| \notag \\
\le & \frac12\sum_i\int_{T_i}\left| \frac{G^{(p+2)}(\xi_i;x_k)}{(p+2)!}(s-s_{i,*})\prod\limits_{j=0}^{p}(s-s_{i,j}) \right| w(s)\gamma(s)ds \notag\\
%\le& C\sum_i\int_{T_i}\left| u^{(p+4)}(x_k)\xi_s(s-s_{i,*})\prod\limits_{j=0}^{p}(s-s_{i,j}) \right| w(s)\gamma(s)ds \notag\\
%\le&C h^{p+2}|u^{(p+4)}|_{\infty} \int_{B_{\delta}(0)}s^2\gamma(s)ds.
%\le& C(p)\delta |u^{(p+4)}|_\infty\sum_i\int_{T_i}\left| (s-s_{i,*})\prod\limits_{j=0}^{p}(s-s_{i,j}) \right| w(s)\gamma(s)ds \notag\\
\le&C(p)\delta h^{p+2}|u^{(p+4)}|_{\infty} \int_{B_{\delta}(0)}w(s)\gamma(s)ds.
\end{align}
This completes the proof.
The proof of the two-dimensional case is similar and we omit it here.
% 二维情况证明类似，这里省略其证明。
}

\bibliographystyle{IMANUM-BIB}
\bibliography{IMANUM-refs}

%\clearpage

\end{document}